\documentclass[11pt,a4paper]{article}

\usepackage[a4paper,margin=30mm]{geometry}
\usepackage{microtype}
\usepackage{amsmath,amssymb,amsfonts,amsthm,mathtools}
\usepackage{mathrsfs}
\usepackage{bm}
\usepackage{xcolor}
\usepackage{comment}
\usepackage{graphicx}
\usepackage{booktabs}
\usepackage[numbers,sort&compress]{natbib}
\usepackage[hidelinks]{hyperref}
\usepackage{enumitem}
\usepackage{units}
\usepackage{subcaption}
\usepackage{epsfig}

\setlist{nosep}
\DeclareMathOperator{\Var}{Var}

\theoremstyle{plain}
\newtheorem{theorem}{Theorem}[section]
\newtheorem{lemma}[theorem]{Lemma}

\newtheorem{proposition}[theorem]{Proposition}
\newtheorem{definition}[theorem]{Definition}
\newtheorem{remark}[theorem]{Remark}

\colorlet{revisioncolor}{black}

\newcommand{\safeincludegraphics}[2][]{%
  \IfFileExists{#2}{\includegraphics[#1]{#2}}{%
    \fbox{\parbox[c][0.24\textheight][c]{0.82\linewidth}{\centering
    Figure file not provided:\\[2mm]\texttt{\detokenize{#2}}}}%
  }%
}

\title{\bfseries Optimal history encoding for elastic-plastic hereditary laws: Sharp input and constitutive approximation}

\author{
  M.~Ortiz${}^\dagger$ and P.~Pedregal${}^\ddagger$ \\
  \small ${}^\dagger$California Institute of Technology, Engineering and Applied Science Division,
  \\ \small Pasadena, CA 91125, USA
  \\
  \small ${}^\dagger$Centre Internacional de M\`etodes Num\`erics a l'Enginyeria (CIMNE),
  \\ \small Universitat Polit\`ecnica de Catalunya, Jordi Girona 1, 08034 Barcelona, Spain
  \\
  \small ${}^\ddagger$Escuela T\'ecnica Superior de Ingenieros Industriales (ETSII) y Departamento de Matemáticas,
  \\ \small Universidad de Castilla-La Mancha, Av. Camilo José Cela, s/n, 13071 Ciudad Real, Spain.
}

\date{}

\begin{document}
\maketitle

\begin{abstract}
We formulate rate-independent elastic-ideally-plastic response directly in hereditary form and study its approximation by finite history surrogates. At the material-point level, the constitutive law is the vector play operator generated by metric projection onto a closed convex elastic domain in stress space. Starting from the closest-point return mapping for step inputs, we pass to absolutely continuous driving histories, for which the constitutive law admits a differential form: the stress remains in the elastic domain and the difference of rates belongs almost everywhere to the normal cone. In this $W^{1,1}$ setting, the hereditary law is causal, contracts variation, and satisfies a $BV$-to-$L^\infty$ stability estimate. We then approximate histories by right-continuous step surrogates with at most $N$ constant pieces. For absolutely continuous inputs, we prove a sharp minimax theorem for input approximation in $L^\infty$, normalized by the $BV$ norm: the optimal encoder is given by equal-variation sampling. For constitutive approximation, the correct vector-valued minimax statement is obtained by allowing the encoder to be material-law aware: it may compress the exact stress history $\mathcal P(\pi)$ rather than only the driving history $\pi$. The resulting stress-aware encoder, followed by the same discrete hereditary decoder, gives the sharp value $(2N)^{-1}$ under the natural nondegeneracy assumption $0\in\operatorname{int}C$. 
The scalar complementary-variable case is also recorded: there the input equal-variation encoder is sharp because the scalar stop/play operator is $L^\infty$-nonexpansive in the complementary variable. The results identify cumulative variation as the natural variable for sampling and compressing both driving and constitutive histories.
\end{abstract}



\section{Introduction}

The theory of {\sl materials with memory} furnishes the most general representation of inelastic materials, whose response is irreversible and history dependent. According to Rivlin \cite{Rivlin:1972}: ``The characteristic property of inelastic solids which distinguishes them from elastic solids is the fact that the stress measured at time $t$ depends not only on the instantaneous value of the deformation but also on the entire history of deformation." The origins of the theory may be traced to a series of papers by Green and Rivlin starting in 1957 \cite{Green:1957, Green:1959, Green:1960}, who proposed the use of hereditary constitutive laws, originally developed by Boltzmann \cite{Boltzmann:1874} and Volterra \cite{Volterra:1909} in the linear case, for the description of non-linear viscoelastic materials \cite{Rivlin:1955}. The hereditary functional approach to inelasticity was introduced into thermodynamics by Coleman \cite{Coleman:1964a}. The foundations underlying the memory-functional and the internal-variable formalisms were critically reviewed by Kestin and Rice \cite{Kestin:1970}. The correspondence and, in some cases, equivalence between the material-with-memory, internal variable and differential formulations of inelasticity have also been extensively investigated \cite{Coleman:1967, Valanis:1967, Lubliner:1969, Lubliner:1973}.

In the history view of inelasticity, the material response is characterized by a \emph{history law}, i.~e, an \emph{operator} that maps histories of strain to histories of stress. The mapping may be linear, as in linear viscoelasticity, or nonlinear, as in plasticity, and must fulfill physical principles such as \emph{causality}, or dependence on past history only, and the \emph{second law} of thermodynamics. In this framework, a fundamental question concerns the representation and approximation of inelastic materials, by \emph{finite-rank} hereditary operators, in a sense to be made precise.  

For linear operators, such as arise, e.~g., in linear viscoelasticity, this problem falls squarely within the theory of {\sl $N$-widths} \cite{Pinkus:1985, ortiz2025linear} and was solved by Schmidt as early as 1907 \cite{Schmidt:1907}. The appeal of the theory of $N$-widths is that it identifies subspaces of histories of given dimension resulting in the best possible approximation of a class of hereditary laws. Furthermore, the approximation of linear hereditary laws by finite-rank operators is equivalent to the formulation of viscoelastic models in terms of a finite number of {\sl history} or {\sl internal variables}. The theory of $N$-widths thus also answers the question of what is the best choice of history or internal variables for purposes of representing a given class of linear viscoelastic materials \cite{Liu:2023, Bhattacharya:2023, ortiz2025linear}. 

The nonlinear setting is more subtle. Here, one no longer seeks a finite-dimensional linear subspace of histories, but rather a nonlinear representation of the constitutive map under explicit information constraints. This goal invites an \emph{encoder/decoder} interpretation: the encoder compresses the driving history into a finite surrogate, and the decoder reconstructs the stress history from that surrogate \cite{NovakWozniakowski2008, NovakWozniakowski2010, NovakWozniakowski2012}. Encoder/decoder schemes originate in \emph{information-based complexity} theory \cite{MicchelliRivlin1977, TraubWozniakowski1980, TraubWasilkowskiWozniakowski1988}, and have been extensively used in a number of areas of application, including vector quantization, adaptive approximation, piecewise polynomial and tree approximation and neural network representations, among others \cite{ TraubWasilkowskiWozniakowski1988, DeVoreLorentz1993, DeVore1998, BalanCasazzaEdidin2000, CotterDashtiRobinsonStuart2010, DeVoreHowardMicchelli2011, SchwabStuart2012, SeidmanKissasPappasPerdikaris2023}. In mechanics, encoder/decoder representations provide a unified abstraction for reduced-order modeling and modern representation learning \cite{TraubWozniakowski1980, BhattacharyaHosseiniKovachkiStuart2021}.

The present work is specifically concerned with rate-independent elastic-ideally-plastic materials. The first objective is to reformulate the classical incremental constitutive law \cite{Moreau1977, OrtizHerrera1981, MonteiroMarques1993, KunzeMonteiroMarques2000} directly in hereditary form on absolutely continuous driving histories, starting from the closest-point projection time-discretization \cite{OrtizHerrera1981, OrtizPinskyTaylor1983, OrtizMartin1989}. At the material-point level, this leads to the vector play operator generated by the elastic domain and to a clean differential formulation in which the stress remains in the elastic domain and the difference of rates belongs almost everywhere to the corresponding normal cone. In this $W^{1,1}$ setting, the hereditary law is causal, contracts total variation, and satisfies a $BV$-to-$L^\infty$ stability estimate.

The second objective is to determine optimal finite encodings by right-continuous step histories with at most $N$ constant pieces. There are two distinct, but complementary, questions. The first is the purely input-side problem: how well can the driving history $\pi=\mathbb C\epsilon$ itself be approximated in $L^\infty$ under a $BV$ normalization? We prove a sharp minimax theorem on absolutely continuous histories, and the optimal rule is equal-variation sampling. The second is the constitutive problem: how well can the stress history $\mathcal P(\pi)$ be represented after decoding through the same hereditary law? Here the encoder/decoder viewpoint is essential. If the encoder is constrained to be an input approximant and the decoder is controlled only by its vector $BV$-to-$L^\infty$ stability, the available estimates are suboptimal. If, however, the encoder is allowed to be material-law aware, as permitted by the minimax formulation of the constitutive error, then it can compress the stress history itself. Since step histories taking values in the elastic domain are fixed points of the decoder, equal-variation sampling of $\mathcal P(\pi)$ yields the sharp vector constitutive value $(2N)^{-1}$, under the natural assumption $0\in\operatorname{int}C$. This distinction keeps the encoder/decoder interpretation intact while separating an optimal constitutive code from a merely input-first code.

The paper is organized as follows. Section~\ref{sec:hereditary} formulates the elastic-plastic history operator on step inputs, extends it to absolutely continuous histories in differential form, and establishes its basic well-posedness and stability properties. Section~\ref{sec:optimal} introduces the admissible class of $N$-piece surrogates, proves the sharp minimax theorem for input approximation and proves the sharp vector constitutive minimax theorem for stress-aware encoders. The scalar complementary-variable result is then recorded as the case in which input-side $L^\infty$ control does propagate through the decoder. Section~\ref{sec:numerical} provides proportional and nonproportional Mises illustrations that distinguish input-first encodings from constitutive stress encodings. Section~\ref{sec:conclusions} summarizes the results and discusses their implications for Data-Driven sampling of stress/strain histories and for global solution strategies based on equivalent body forces.

\section{Elastic-plastic hereditary laws}
\label{sec:hereditary}

An axiomatic foundation for rate-independent elastoplasticity can be built on the principle of maximum dissipation \cite{Lubliner1989, Lubliner1990, Lubliner2008}. Plasticity models are conventionally formulated in rate form. \emph{Ideally-plastic} \emph{rate-independent} \emph{elastic-plastic} solids are then fully described by \emph{Hooke's law} and an \emph{elastic domain} in stress space, which can be unbounded (e.~g., isochoric plasticity) and non-smooth (e.~g., crystal plasticity), and is commonly observed to be \emph{convex} (see \cite{OrtizPopov1983} for a compilation of early experimental work). For such materials, the principle of maximum dissipation finds mathematical expression in the concept of \emph{subdifferential}, introduced by Moreau in his pioneering work \cite{Moreau1963} and subsequently applied to the study of plastic and viscoplastic materials (see \cite{Moreau1970, Moreau1971a, Moreau1974, Moreau1976, OrtizHerrera1981, OrtizPinskyTaylor1983} for some early references that connect plasticity and convex analysis). Moreau's formalism also makes contact with the theory of semigroups generated by subdifferential operators and time-discretizations thereof \cite{Brezis1973, OrtizHerrera1981}. For present purposes, it is more convenient to formulate plasticity directly in \emph{hereditary form}, as shown next.

The aim of this section is twofold. First, we state the elastic-plastic law in a form that is entirely hereditary and therefore independent of any rate decomposition. Second, we establish the stability properties of that law on absolutely continuous histories, to be used subsequently for purposes of optimal approximation.

\subsection{Setting}

Let $E=\mathbb{R}^n$ be a finite-dimensional \emph{strain space} and let $F=E^\ast$ be its dual \emph{stress space}. We metrize $E$ by a symmetric positive-definite elastic modulus $\mathbb{C}\in L(E,F)$ and $F$ by its inverse, to wit:
\begin{align}
    \langle \epsilon^1,\epsilon^2\rangle_E &:= \mathbb{C}\epsilon^1\cdot \epsilon^2,
    &
    \|\epsilon\|_E &:= \sqrt{\mathbb{C}\epsilon\cdot\epsilon},
    \\
    \langle \sigma^1,\sigma^2\rangle_F &:= \mathbb{C}^{-1}\sigma^1\cdot \sigma^2,
    &
    \|\sigma\|_F &:= \sqrt{\mathbb{C}^{-1}\sigma\cdot\sigma}.
\end{align}
The duality pairing between stress and strain is
\begin{equation}
  \langle \sigma,\epsilon\rangle := \sigma\cdot\epsilon .
\end{equation}
For definiteness, one may set $E=F=\mathbb{R}^{3\times 3}_{\mathrm{sym}}$, corresponding to three-dimensional linearized-kinematics plasticity, but the analysis that follows is independent of dimension. Throughout the paper, we work with the \emph{driving} and \emph{complementary variable} histories
\begin{equation} \label{eq:driver_and_comp}
    \pi(t):=\mathbb{C}\epsilon(t) \in F,
    \quad
    w(t):=\pi(t)-\sigma(t) := \mathbb{C} \epsilon^p(t) \in F,
\end{equation}
respectively, rather than directly with the history of strain $\epsilon(t)$ or plastic strain $\epsilon^p(t)$. This choice makes the constitutive law naturally stress-valued. The inverse relation is $\epsilon=\mathbb C^{-1}\pi$, and $w$ is the stress-valued counterpart of the plastic strain.

Let $C\subset F$ be a non-empty, closed, convex \emph{elastic domain}, not necessarily bounded and not necessarily smooth. Its metric projection,
\begin{equation}
    P_C(\sigma):=\arg\min_{\eta\in C}\|\sigma-\eta\|_F, \quad \sigma\in F,
\end{equation}
is well defined, $1$-Lipschitz and non-expansive \cite{Rockafellar:1970, BauschkeCombettes2017},
\begin{equation} \label{eq:proj-lip}
    \|P_C(\sigma^1)-P_C(\sigma^2)\|_F\le \|\sigma^1-\sigma^2\|_F,
    \quad \forall \sigma^1,\sigma^2\in F.
\end{equation}
Moreover,
\begin{equation} \label{eq:proj-normal}
    \sigma-P_C(\sigma)\in N_C(P_C(\sigma)),
\end{equation}
where $N_C(\sigma)$ denotes the \emph{normal cone} to $C$ at $\sigma$. In addition, if $\nu^i \in N_C(\sigma^i)$, $i=1,2$, then
\begin{equation} \label{eq:normal-monotonicity}
    \langle \nu^1-\nu^2,\sigma^1-\sigma^2\rangle_F\ge 0 ,
\end{equation}
which constitutes a monotonicity property. 

Let $H$ be a real Hilbert space. We denote by $BV([0,T];H)$ the space of $H$-valued functions of bounded variation on $[0,T]$, with
\begin{equation}
    \Var(u;[0,T])
    :=
    \sup_{\Pi}\sum_{j=1}^{m}\|u(t_j)-u(t_{j-1})\|_H,
\end{equation}
the supremum being taken over all partitions $\Pi=\{0=t_0<t_1<\cdots<t_m=T\}$ \cite{AmbrosioFuscoPallara2000}. 

An elementary property of the variation is domain-additivity.

\begin{lemma}[Additivity of the variation on subintervals] \label{lem:variation-additivity}
Let $u\in BV([0,T];H)$ and let $0\le a\le b\le c\le T$. Then
\begin{equation}
\Var(u;[a,c])=\Var(u;[a,b])+\Var(u;[b,c]).
\end{equation}
\end{lemma}

\begin{proof}
For any partitions
\begin{equation}
\Pi_1=\{a=s_0<\cdots<s_m=b\},\quad
\Pi_2=\{b=r_0<\cdots<r_n=c\},
\end{equation}
their union is a partition of $[a,c]$, hence
\begin{equation}
\sum_{i=1}^m \|u(s_i)-u(s_{i-1})\|_H
+\sum_{j=1}^n \|u(r_j)-u(r_{j-1})\|_H
\le \Var(u;[a,c]).
\end{equation}
Taking the supremum over $\Pi_1$ and $\Pi_2$ gives
\begin{equation}
\Var(u;[a,b])+\Var(u;[b,c])\le \Var(u;[a,c]).
\end{equation}
Conversely, let
\begin{equation}
\Pi=\{a=t_0<\cdots<t_\ell=c\}
\end{equation}
be any partition of $[a,c]$. By adding the point $b$ to $\Pi$ if necessary, we obtain a refinement
that splits into a partition of $[a,b]$ and a partition of $[b,c]$. Therefore,
\begin{equation}
\sum_{k=1}^\ell \|u(t_k)-u(t_{k-1})\|_H
\le \Var(u;[a,b])+\Var(u;[b,c]).
\end{equation}
Taking the supremum over all partitions $\Pi$ yields
\begin{equation}
\Var(u;[a,c])\le \Var(u;[a,b])+\Var(u;[b,c]).
\end{equation}
Combining the two inequalities proves the claim.
\end{proof}

The following lemma addresses a case that will arise often. 

\begin{lemma}[Variation of a step function] \label{lem:variation-step-function}
Let $H$ be a real Hilbert space, let
\begin{equation}
\Pi=\{0=t_0<t_1<\cdots<t_N=T\},
\end{equation}
and let $u:[0,T]\to H$ be the right-continuous step function defined by
\begin{equation}
u(t)=u_{k-1}\quad \text{for } t\in [t_{k-1},t_k),\quad k=1,\dots,N,
\end{equation}
and $u(T)=u_N$, where $u_0,\dots,u_N\in H$. Then
\begin{equation}
\Var(u;[0,T])=\sum_{k=1}^N \|u_k-u_{k-1}\|_H.
\end{equation}
\end{lemma}

\begin{proof}
For the partition $\Pi$ itself,
\begin{equation}
\sum_{k=1}^N \|u(t_k)-u(t_{k-1})\|_H
=\sum_{k=1}^N \|u_k-u_{k-1}\|_H,
\end{equation}
hence
\begin{equation}
\Var(u;[0,T])\ge \sum_{k=1}^N \|u_k-u_{k-1}\|_H.
\end{equation}
Conversely, let $\widetilde\Pi=\{0=s_0<\cdots<s_m=T\}$ be any partition of $[0,T]$.
Since $u$ is constant on each interval $[t_{k-1},t_k)$, the only contributions to
\begin{equation}
\sum_{j=1}^m \|u(s_j)-u(s_{j-1})\|_H
\end{equation}
come from crossings of the jump points $t_1,\dots,t_{N-1}$ and possibly the endpoint $T$.
By the triangle inequality, the total contribution across the $k$-th jump is at most
$\|u_k-u_{k-1}\|_H$. Summing over all jumps gives
\begin{equation}
\sum_{j=1}^m \|u(s_j)-u(s_{j-1})\|_H
\le
\sum_{k=1}^N \|u_k-u_{k-1}\|_H.
\end{equation}
Taking the supremum over all partitions $\widetilde\Pi$ proves the claim.
\end{proof}

For purposes of analysis, we shall use the equivalent norms
\begin{equation}
    \|u\|_{BV,0}:=\|u(0)\|_H+\Var(u;[0,T]),
    \quad
    \|u\|_{BV,\infty}:=\|u\|_{L^\infty([0,T];H)}+\Var(u;[0,T]).
\end{equation}
We recall that every $BV$ function admits left and right limits everywhere and, henceforth, we identify it with its right-continuous representative \cite{AmbrosioFuscoPallara2000, Brezis2011, EvansGariepy2015}. For absolutely continuous histories we use the standard Sobolev space $W^{1,1}(0,T;H) \hookrightarrow BV([0,T];H)$, characterized by
\begin{equation}
    u(t)=u(0)+\int_0^t \dot u(s)\,ds,
    \quad
    \Var(u;[0,T])=\int_0^T \|\dot u(s)\|_H\,ds.
\end{equation}
In particular, every $W^{1,1}$ history has an absolutely continuous representative, and all the variation is carried by its density $\dot u$.

The following elementary approximation lemma will be used repeatedly.

\begin{lemma}[Uniform step approximation with controlled variation] \label{lem:step-approx}
For every $u\in BV([0,T];H)$ and every $N\in\mathbb{N}$ there exists a right-continuous step function $u_N$ with at most $N$ constant pieces such that
\begin{equation}
    \|u-u_N\|_{L^\infty([0,T];H)}\le \frac{\Var(u;[0,T])}{N},
    \quad
    \Var(u_N;[0,T])\le \Var(u;[0,T]).
\end{equation}
\end{lemma}

\begin{proof}
Let $\alpha(t):=\Var(u;[0,t])$ and set $V:=\Var(u;[0,T])$. If $V=0$, the constant function $u_N\equiv u(0)$ does the job. Assume $V>0$. For $k=0,\dots,N-1$ let
\begin{equation}
    t_k:=\inf\{t\in[0,T]:\alpha(t)\ge kV/N\},
\end{equation}
with $t_0=0$ and $t_N=T$. Define a right-continuous step function $\widetilde u_N$ by
\begin{equation}
    \widetilde u_N(t):=u(t_{k-1}),
    \quad
    t\in [t_{k-1},t_k), \quad k=1,\dots,N,
\end{equation}
and $\widetilde u_N(T):=u(T)$. If some consecutive nodes coincide, the corresponding intervals are empty; removing them yields a right-continuous step function $u_N$ with at most $N$ non-empty pieces and the same values as $\widetilde u_N$. For every $t\in[t_{k-1},t_k)$,
\begin{equation}
    \|u(t)-u_N(t)\|_H
    \le \Var(u;[t_{k-1},t])
    \le \alpha(t_k^-)-\alpha(t_{k-1})
    \le \frac{V}{N},
\end{equation}
because $\alpha(t_k^-)\le kV/N$ by minimality of $t_k$. This proves the uniform estimate, while at $t=T$ the error is zero by construction. Finally,
\begin{equation}
    \Var(u_N;[0,T])
    =\sum_{k:\,t_k>t_{k-1}}\|u(t_k)-u(t_{k-1})\|_H
    \le \sum_{k=1}^{N}\Var(u;[t_{k-1},t_k])
    \le \Var(u;[0,T]) ,
\end{equation}
where the equality is specific to step functions, the first inequality follows from the definition of the total variation and the second inequality follows its interval additivity, Lemma~\ref{lem:variation-additivity}.
\end{proof}

\subsection{Discrete hereditary law on step inputs}

\begin{figure}[ht]
\begin{center}
	\begin{subfigure}{0.34\textwidth}\caption{} \safeincludegraphics[width=0.99\linewidth]{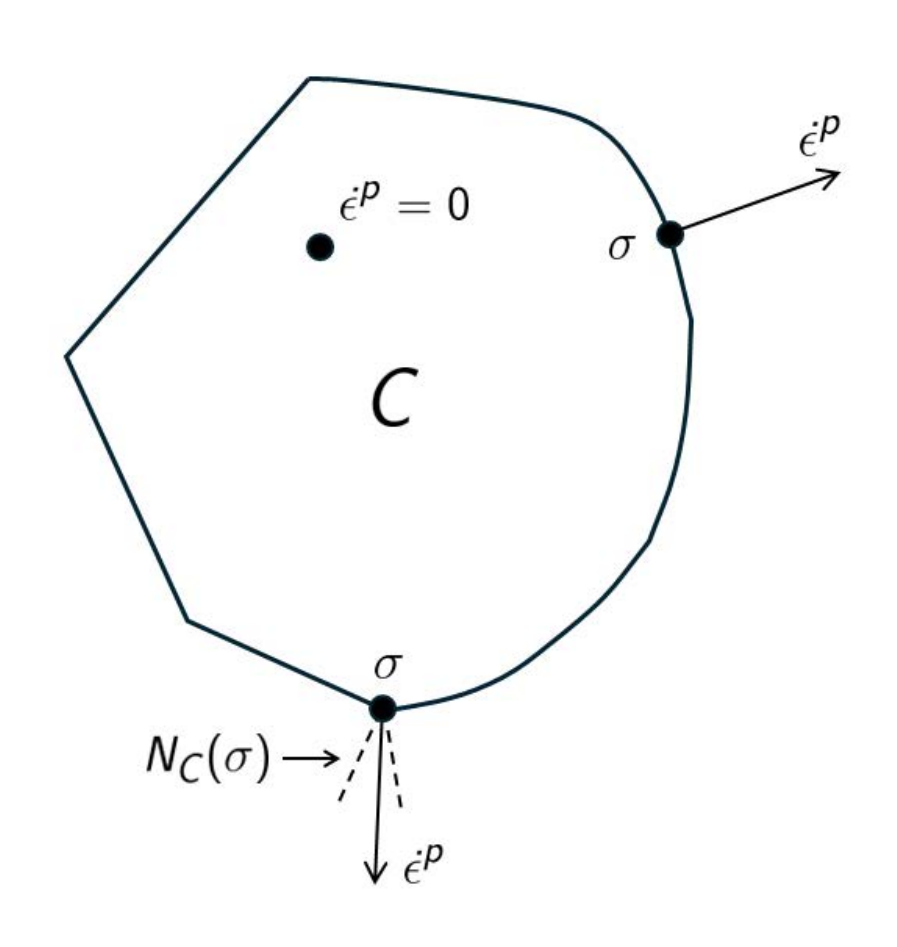}
	\end{subfigure}
	\begin{subfigure}{0.55\textwidth}\caption{} \safeincludegraphics[width=0.99\linewidth]{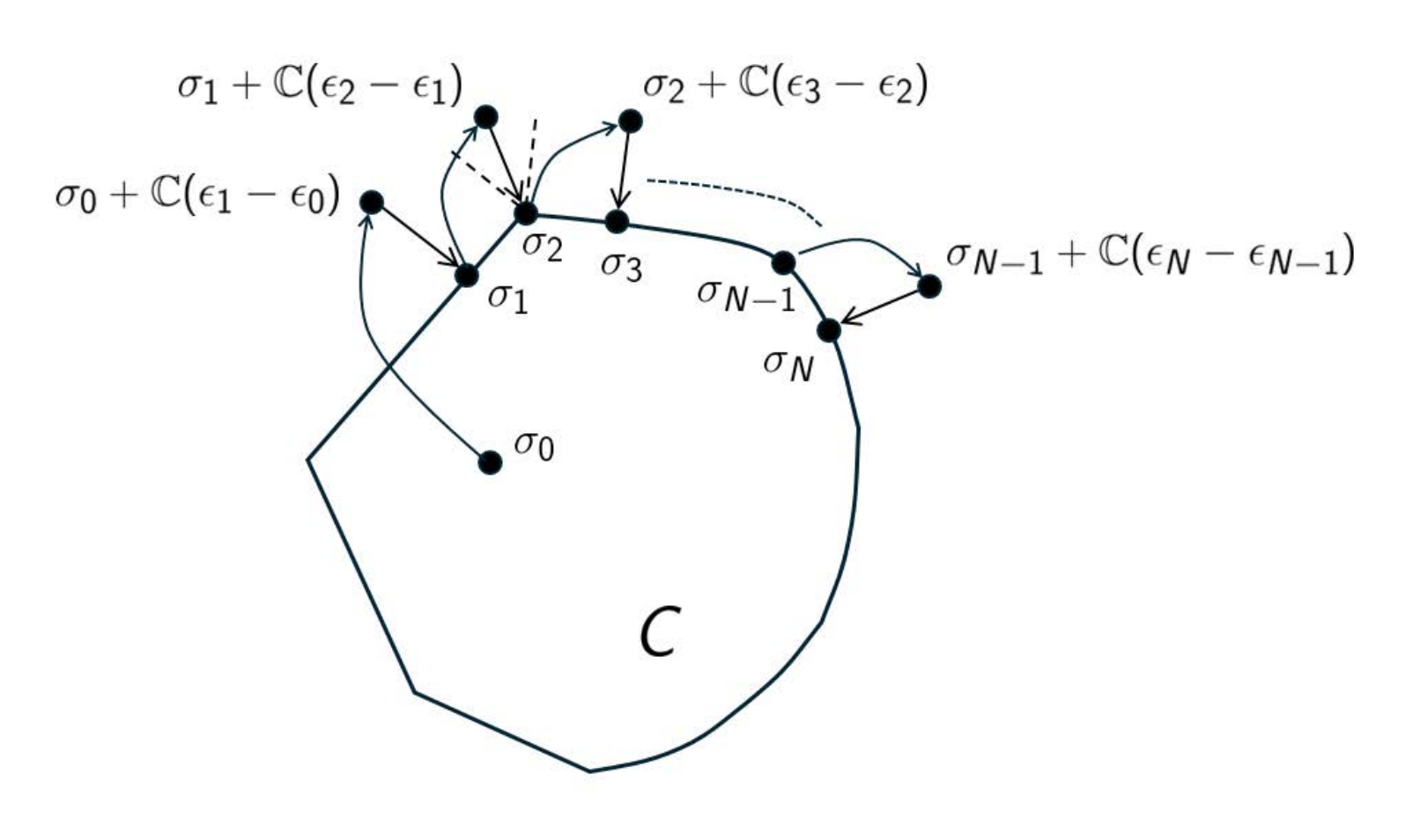}
	\end{subfigure}
    \caption{Schematic representation of the constitutive response of a rate-independent elastic-ideally plastic solid. a) Normality-based plastic flow rule $\dot{\epsilon}^p \in N_C(\sigma)$ \cite{Moreau1963}. Here $C$ is the elastic domain, $\sigma$ the current stress, $\dot{\epsilon}^p$ the instantaneous plastic-strain rate and $N_C(\sigma)$ is the normal cone at $\sigma$. b) Resolvent time discretization of the rate stress-strain relation. At every step, the stress is updated by first taking an elastic-predictor step and then projecting the result to the nearest point on the elastic domain \cite{OrtizHerrera1981}.}
    \label{fig:plasticity}
\end{center}
\end{figure}

We begin by defining the hereditary law on step histories, see Fig.~\ref{fig:plasticity}. Let
\begin{equation}
    \Pi=\{0=t_0<t_1<\cdots<t_M=T\} 
\end{equation}
be a partition, and let $\pi:[0,T]\to F$ be right-continuous and constant on every interval $[t_{k-1},t_k)$, with nodal values $\pi_k:=\pi(t_k)$, $k=0,\dots,M$. The corresponding stress history $\sigma:[0,T]\to C$ is defined recursively as
\begin{equation} \label{eq:discrete-play}
    \sigma_0:=P_C(\pi_0),
    \quad
    \sigma_k:=P_C\bigl(\sigma_{k-1}+\pi_k-\pi_{k-1}\bigr),
    \quad 
    k=1,\dots,M.
\end{equation}
We then set $\sigma(t):=\sigma_{k-1}$ for $t\in [t_{k-1},t_k)$, $k=1,\dots,M$, and $\sigma(T):=\sigma_M$.

Equation \eqref{eq:discrete-play} is simultaneously the resolvent discretization of the sweeping process \cite{Moreau1977, OrtizHerrera1981, MonteiroMarques1993, KunzeMonteiroMarques2000} and the classical closest-point return mapping of computational plasticity \cite{OrtizPinskyTaylor1983, OrtizMartin1989}. In particular, the discrete law is defined for arbitrary closed convex, possibly nonsmooth, elastic domains.

A first elementary point is that the discrete output does not depend on redundant nodes.

\begin{lemma}[Refinement invariance] \label{lem:refinement}
Let $\pi$ be a step input on a partition $\Pi$ and let $\widetilde\Pi$ be a refinement of $\Pi$. If $\pi$ is regarded as a step input on $\widetilde\Pi$ by repeating the same values on the refined subintervals, then the stress history produced by \eqref{eq:discrete-play} is unchanged.
\end{lemma}

\begin{proof}
It suffices to insert a single extra node $\tau\in(t_{k-1},t_k)$ and to note that the two corresponding increments are
\begin{equation}
\pi(\tau)-\pi(t_{k-1})=0,
\quad
\pi(t_k)-\pi(\tau)=\pi(t_k)-\pi(t_{k-1}).
\end{equation}
Hence the first refined update leaves the stress unchanged and the second reproduces the original update. Iterating proves the claim.
\end{proof}

The discrete evolution enjoys the expected range, variation, and stability properties.

\begin{theorem}[Discrete stability]
\label{thm:discrete-stability}
Let $\pi$ be a step input and let $\sigma$ be defined by \eqref{eq:discrete-play}. Then:
\begin{enumerate}
\item[(i)] $\sigma_k\in C$ for every $k=0,\dots,M$.
\item[(ii)] $\displaystyle \Var(\sigma;[0,T])\le \Var(\pi;[0,T])$.
\item[(iii)] If $\pi^1,\pi^2$ are step inputs on a common partition and $\sigma^1,\sigma^2$ are the corresponding outputs, then
\begin{equation}
  \|\sigma^1-\sigma^2\|_{L^\infty([0,T];F)}
  \le
  \|\pi^1(0)-\pi^2(0)\|_F+\Var(\pi^1-\pi^2;[0,T]).
  \label{eq:discrete-bv-stability}
\end{equation}
In particular,
\begin{equation}
\|\sigma^1-\sigma^2\|_{L^\infty([0,T];F)}
\le
\|\pi^1-\pi^2\|_{BV,0}.
\end{equation}
\end{enumerate}
\end{theorem}

\begin{proof}
Assertion (i) is immediate from the definition of the projection.

For (ii), by \eqref{eq:proj-lip},
\begin{equation}
  \|\sigma_k-\sigma_{k-1}\|_F
  =
  \left\|P_C\bigl(\sigma_{k-1}+\pi_k-\pi_{k-1}\bigr)-P_C(\sigma_{k-1})\right\|_F
  \le
  \|\pi_k-\pi_{k-1}\|_F .
\end{equation}
Summation over $k$ and Lemma~\ref{lem:variation-step-function} yield the result.

For (iii), set $\delta_k:=\pi_k^1-\pi_k^2$ and $e_k:=\sigma_k^1-\sigma_k^2$. By \eqref{eq:proj-lip},
\begin{equation}
\|e_0\|_F
=
\|P_C(\pi_0^1)-P_C(\pi_0^2)\|_F
\le
\|\delta_0\|_F,
\end{equation}
and, for $k\ge 1$,
\begin{equation}
\begin{split}
\|e_k\|_F
&=
\left\|
P_C\bigl(\sigma_{k-1}^1+\pi_k^1-\pi_{k-1}^1\bigr)
-
P_C\bigl(\sigma_{k-1}^2+\pi_k^2-\pi_{k-1}^2\bigr)
\right\|_F
\\
&\le
\|e_{k-1}+\delta_k-\delta_{k-1}\|_F
\le
\|e_{k-1}\|_F+\|\delta_k-\delta_{k-1}\|_F .
\end{split}
\end{equation}
Induction therefore gives
\begin{equation}
\|e_k\|_F
\le
\|\delta_0\|_F+\sum_{j=1}^k\|\delta_j-\delta_{j-1}\|_F
\le
\|\pi^1(0)-\pi^2(0)\|_F+\Var(\pi^1-\pi^2;[0,T]).
\end{equation}
Taking the maximum over $k$ proves \eqref{eq:discrete-bv-stability}.
\end{proof}

\begin{remark}[Failure of $L^\infty$ non-expansiveness in the vector case]\label{rem:failure-Linfty}
{\rm When $\dim F>1$, the stronger estimate
\begin{equation}
\|\sigma^1-\sigma^2\|_{L^\infty}
\le
\|\pi^1-\pi^2\|_{L^\infty}
\end{equation}
is false in general. For instance, let $F=\mathbb{R}^2$, let $C=[-1,1]^2$, and consider the two-step inputs
\begin{equation}
\pi_0^1=\pi_1^1=(-2,-2), \quad \pi_2^1=(-2,-1),
\end{equation}
\begin{equation}
\pi_0^2=\pi_1^2=(-2,-1), \quad \pi_2^2=(-1,-1).
\end{equation}
Then
\begin{equation}
\sup_{k=0,1,2}\|\pi_k^1-\pi_k^2\|_2=1,
\end{equation}
whereas the discrete outputs are
\begin{equation}
\sigma_0^1=\sigma_1^1=(-1,-1), \quad \sigma_2^1=(-1,0),
\end{equation}
\begin{equation}
\sigma_0^2=\sigma_1^2=(-1,-1), \quad \sigma_2^2=(0,-1),
\end{equation}
so that
\begin{equation}
\sup_{k=0,1,2}\|\sigma_k^1-\sigma_k^2\|_2=\sqrt{2}.
\end{equation}
Thus the vector play/stop update is stable in the $BV$ sense of \eqref{eq:discrete-bv-stability}, but not non-expansive in $L^\infty$.
} \hfill$\square$
\end{remark}

\begin{remark}[Scalar case]\label{rem:scalar-case} {\rm The scalar case is special. When $\dim F=1$, we may identify $F$ with $\mathbb{R}$ and represent the elastic domain as an interval $C=[\alpha,\beta]$. Then, for two scalar step inputs $\pi^1,\pi^2$ on a common partition, the corresponding complementary variable outputs $w^1,w^2$, see (\ref{eq:driver_and_comp}), satisfy the standard maximum-norm estimate
\begin{equation} \label{eq:maximum-norm-estimate}
\|w^1-w^2\|_{L^\infty(0,T)}
\le
\|\pi^1-\pi^2\|_{L^\infty(0,T)}.
\end{equation}
Indeed, in the scalar case the recursion is
\begin{equation}
w_k^i=\operatorname{med}\{\pi_k^i-\beta,\;w_{k-1}^i,\;\pi_k^i-\alpha\},\quad
w_0^i=\operatorname{med}\{\pi_0^i-\beta,\;0,\;\pi_0^i-\alpha\} ,
\end{equation}
where $\operatorname{med}$ is the median of three real numbers, i.~e., $\operatorname{med}\{a,b,c\}$ is the number among $a$, $b$ and $c$ that lies between the other two, and we use the identity $\operatorname{proj}_{[a,c]}(b) = \operatorname{med}\{a,b,c\}$. Since $\operatorname{med}$ is $1$-Lipschitz in $\ell^\infty$, we get
\begin{equation}
\|w_0^1-w_0^2\|\le \|\pi^1-\pi^2\|_{L^\infty(0,T)},
\end{equation}
and
\begin{equation}
\|w_k^1-w_k^2\|\le \max\{\|\pi_k^1-\pi_k^2\|,\;\|w_{k-1}^1-w_{k-1}^2\|\}\le \|\pi^1-\pi^2\|_{L^\infty(0,T)}.
\end{equation}
Hence, by induction, $\|w_k^1-w_k^2\|\le \|\pi^1-\pi^2\|_{L^\infty(0,T)}$ for all $k$, and, taking the maximum in $k$ (\ref{eq:maximum-norm-estimate}) follows. The conclusion extends to two arbitrary scalar step inputs $\pi^1,\pi^2$ by simply reducing them to a common partition, see Lemma~\ref{lem:refinement}. Thus, scalar $L^\infty$ non-expansiveness holds for the complementary variable. 

By contrast, the stress variable $\sigma$ need not be non-expansive even in the scalar case. For instance, let $C=[-1,1]$, and choose $\pi^1=(-3,-3,-1)$ and $\pi^2 = (-2,-2,-2)$. Then, $\|\pi^1-\pi^2\|_{L^\infty}=1$. Noting that $w_k = \operatorname{med}\{\pi_k-1,\;w_{k-1},\;\pi_k+1\}$, and $\sigma_k=\pi_k-w_k$, with initial value $w_0=\operatorname{med}\{\pi_0-1,\;0,\;\pi_0+1\}$, we compute for $\pi^1$: $w^1=(-2,-2,-2)$ and $\sigma^1=(-1,-1,1)$; and for $\pi^2$: $w^2=(-1,-1,-1)$ and $\sigma^2=(-1,-1,-1)$. Therefore,
\begin{equation}
\|\sigma^1-\sigma^2\|_{L^\infty}=\max\{0,0,2\}=2 ,
\quad
\|\pi^1-\pi^2\|_{L^\infty}=1,
\end{equation}
and the stress map $\pi\mapsto \sigma$ is not $L^\infty$-nonexpansive, as advertised. 
} \hfill$\square$
\end{remark}

Finally, we stress that if a step history $\pi$ takes values in $C$, then the corresponding stress history as defined above through \eqref{eq:discrete-play} leaves $\pi$ unchanged. 

\begin{proposition}\label{rel-proy}
    If $\pi:[0, T]\to C$ is a step input with corresponding output $\sigma$, then $\sigma\equiv\pi$. 
\end{proposition}
\begin{proof}
The proof is elementary and can be organized through induction in $k=0, 1, \dots, M$. It only uses the fact that $P_C$ is a projection onto the convex set $C$, and as such $P_C\circ P_C=P_C$.
\end{proof}

\subsection{Extension to absolutely continuous histories}

The continuous-time history operator on absolutely continuous inputs is the vector play operator associated with the translated moving set $\pi(t)-C$  \cite{Moreau1977, MonteiroMarques1993, KunzeMonteiroMarques2000, Krejci1996, BrokateSprekels1996, Recupero2020}. Here, we consider absolute continuous inputs simply as a matter of convenience, in order to avoid the need for technical jump/completed-graph theory. In the present setting, the constitutive law can be stated directly in differential form, with no jump or completed-graph machinery.

\begin{theorem}[$W^{1,1}$ well-posedness]
\label{thm:bv-well-posedness}
For every $\pi\in W^{1,1}([0,T];F)$ there exists a unique $\sigma\in W^{1,1}([0,T];C)$ such that
\begin{equation} \label{eq:sweeping-state}
    \sigma(0)=P_C(\pi(0)),
    \quad
    \sigma(t)\in C , 
    \quad 
    \text{for every } t\in[0,T],
\end{equation}
and
\begin{equation} \label{eq:sweeping-inclusion}
    \dot\pi(t)-\dot\sigma(t)\in N_C(\sigma(t)) ,
    \quad 
    \text{for a.e.}\ t\in(0,T).
\end{equation}
Equivalently,
\begin{equation} \label{eq:sweeping-variational}
    \langle \dot\pi(t)-\dot\sigma(t),\, z-\sigma(t)\rangle_F\le 0 ,
    \quad 
    \text{for a.e.}\ t\in(0,T),\ \forall z\in C.
\end{equation}
\end{theorem}

Theorem~\ref{thm:bv-well-posedness} is the standard absolutely continuous sweeping-process formulation for the translated convex set $\pi(t)-C$, see \cite{Moreau1977, MonteiroMarques1993, KunzeMonteiroMarques2000, Krejci1996, BrokateSprekels1996, Recupero2020}.

\begin{definition} \label{def:play-operator}
For $\pi\in W^{1,1}([0,T];F)$, the unique history $\sigma$ set forth by Theorem~\ref{thm:bv-well-posedness} defines an operator $\mathcal{P} : W^{1,1}([0,T];F) \to W^{1,1}([0,T];C)$, with the property that $\sigma = \mathcal{P}(\pi)$. We refer to $\mathcal{P}$ as the \emph{elastic-plastic history} operator, or \emph{play} operator, associated with the elastic domain $C$.
\end{definition}

More generally, if $z\in BV([0,T];F)$ is right-continuous, then the classical completed-graph/catching-up construction yields a unique right-continuous output $\mathcal P_{BV}(z)$ \cite{Moreau1977, MonteiroMarques1993, KunzeMonteiroMarques2000, Krejci1996, BrokateSprekels1996, Recupero2020}. For step inputs this output coincides with the discrete closest-point recursion above. If $z\in W^{1,1}([0,T];F)$, then $\mathcal P_{BV}(z)=\mathcal P(z)$. We shall use the following standard stability property of the BV play operator.

\begin{proposition}[$BV$ stability of the completed-graph play operator]
\label{prop:BV-play-stability}
Let $z^1,z^2\in BV([0,T];F)$ be right-continuous. Then
\begin{equation}
\|\mathcal P_{BV}(z^1)-\mathcal P_{BV}(z^2)\|_{L^\infty([0,T];F)}
\le
\|z^1(0)-z^2(0)\|_F+\Var(z^1-z^2;[0,T]).
\label{eq:BV-play-stability}
\end{equation}
Moreover,
\begin{equation}
\Var(\mathcal P_{BV}(z^1);[0,T])\le \Var(z^1;[0,T]).
\end{equation}
\end{proposition}

\begin{proof}
For step inputs, \eqref{eq:BV-play-stability} is exactly Theorem~\ref{thm:discrete-stability}(iii), after replacing two step partitions by a common refinement, and the variation estimate is Theorem~\ref{thm:discrete-stability}(ii). The general right-continuous $BV$ case follows by the standard completed-graph/catching-up construction for multidimensional play operators with arbitrary $BV$ inputs: approximate the completed graphs by step catching-up paths, apply the preceding discrete estimates on common refinements, and pass to the uniform limit of the outputs. This is the BV-continuity property of the play operator proved in the cited references, in particular \cite{Krejci1996, BrokateSprekels1996, Recupero2020}.
\end{proof}

The operator $\mathcal{P}$ inherits the fundamental features of the discrete scheme.

\begin{theorem}[Basic properties of $\mathcal{P}$]
\label{thm:regularity}
The history operator $\mathcal{P}:W^{1,1}([0,T];F)\to W^{1,1}([0,T];C)$ satisfies:
\begin{enumerate}
\item[(i)] \emph{Range preservation:} $\mathcal{P}(\pi)(t)\in C$ for all $t\in[0,T]$.
\item[(ii)] \emph{Causality:} if $\pi^1=\pi^2$ on $[0,t_\ast]$, then $\mathcal{P}(\pi^1)=\mathcal{P}(\pi^2)$ on $[0,t_\ast]$.
\item[(iii)] \emph{$BV$-to-$L^\infty$ stability:}
\begin{equation}
  \|\mathcal{P}(\pi^1)-\mathcal{P}(\pi^2)\|_{L^\infty([0,T];F)}
  \le
  \|\pi^1(0)-\pi^2(0)\|_F+\Var(\pi^1-\pi^2;[0,T]).
  \label{eq:bv-play}
\end{equation}
In particular,
\begin{equation}
\|\mathcal{P}(\pi^1)-\mathcal{P}(\pi^2)\|_{L^\infty([0,T];F)}
\le
\|\pi^1-\pi^2\|_{BV,0}.
\end{equation}
\item[(iv)] \emph{Variation contraction:}
\begin{equation}
  \Var(\mathcal{P}(\pi);[0,T])\le \Var(\pi;[0,T]).
  \label{eq:var-play}
\end{equation}
\end{enumerate}
\end{theorem}

\begin{proof}
Assertions (i) and (ii) are standard consequences of the sweeping-process formulation of the vector play operator; see \cite{Moreau1977, MonteiroMarques1993, KunzeMonteiroMarques2000, Krejci1996, BrokateSprekels1996, Recupero2020}. 

For (iii), set
\begin{equation}
\sigma^i:=\mathcal{P}(\pi^i) ,
\quad
u:=\sigma^1-\sigma^2 ,
\quad
v:=\pi^1-\pi^2 .
\end{equation}
Then, for a.e.~$t\in(0,T)$ ,
\begin{equation}
\dot u=\dot v-\eta^1+\eta^2 ,
\quad
\eta^i:=\dot\pi^i-\dot\sigma^i\in N_C(\sigma^i) .
\end{equation}
and, by the monotonicity of the normal cone, see (\ref{eq:normal-monotonicity}), 
\begin{equation}
\langle \eta^1(t)-\eta^2(t),\, u(t)\rangle_F\ge 0 ,
\quad
\text{a.e.~in}\ (0,T) .
\end{equation}
Therefore,
\begin{equation}
\frac{d}{dt}\frac12\|u(t)\|_F^2
=
\langle \dot u(t),u(t)\rangle_F
\le
\langle \dot v(t),u(t)\rangle_F
\le
\|\dot v(t)\|_F\,\|u(t)\|_F ,
\quad
\text{a.e.~in}\ (0,T) .
\end{equation}
By the standard differential inequality for the norm of an absolutely continuous map \cite{Brezis2011},
\begin{equation}
\frac{d}{dt}\|u(t)\|_F\le \|\dot v(t)\|_F ,
\quad
\text{a.e.~in}\ (0,T) ,
\end{equation}
integration yields
\begin{equation}
\|u(t)\|_F
\le
\|u(0)\|_F+\int_0^t\|\dot v(s)\|_F\,ds
\le
\|u(0)\|_F+\Var(v;[0,T]).
\end{equation}
Taking the supremum over $t\in[0,T]$ and using $\sigma^i(0)=P_C(\pi^i(0))$ together with the $1$-Lipschitz continuity of $P_C$ proves \eqref{eq:bv-play}.

For (iv), fix $\pi\in W^{1,1}([0,T];F)$ and set $\sigma:=\mathcal P(\pi)$ and $\eta:=\dot\pi-\dot\sigma\in N_C(\sigma)$ a.e.~in $(0,T)$. Since $\sigma\in W^{1,1}([0,T];F)$ and $\sigma(t)\in C$ everywhere in $[0,T]$, at every time $t$ where $\sigma$ is differentiable we have
\begin{equation} \label{eq:two-sided-tangent}
    \dot\sigma(t)\in T_C(\sigma(t))
    \quad\text{and}\quad
    -\dot\sigma(t)\in T_C(\sigma(t)),
\end{equation}
where $T_C(\sigma(t))$ denotes the tangent cone to $C$ at $\sigma(t)$, see \cite{Rockafellar:1970}. Indeed, if $h_j\downarrow 0$, then
\begin{equation}
\frac{\sigma(t+h_j)-\sigma(t)}{h_j}\in \frac{C-\sigma(t)}{h_j},
\quad
\frac{\sigma(t-h_j)-\sigma(t)}{h_j}\in \frac{C-\sigma(t)}{h_j},
\end{equation}
and both sequences converge to $\dot\sigma(t)$ and $-\dot\sigma(t)$, respectively. Since $\eta(t)\in N_C(\sigma(t))=T_C(\sigma(t))^{\circ}$, the polar cone of $T_C(\sigma(t))$ \cite{Rockafellar:1970}, \eqref{eq:two-sided-tangent} implies
\begin{equation}
\langle \eta(t),\dot\sigma(t)\rangle_F\le 0
\quad\text{and}\quad
\langle \eta(t),-\dot\sigma(t)\rangle_F\le 0 ,
\end{equation}
hence,
\begin{equation}
    \langle \eta(t),\dot\sigma(t)\rangle_F=0 .
\end{equation}
Therefore, 
\begin{equation}
\|\dot\sigma(t)\|_F^2
=
\langle \dot\sigma(t),\dot\pi(t)-\eta(t)\rangle_F
=
\langle \dot\sigma(t),\dot\pi(t)\rangle_F
\le
\|\dot\sigma(t)\|_F\,\|\dot\pi(t)\|_F,
\end{equation}
and
\begin{equation}
\|\dot\sigma(t)\|_F\le \|\dot\pi(t)\|_F ,
\end{equation} 
a.e.~in $(0,T)$. Integrating in time finally yields
\begin{equation}
\Var(\sigma;[0,T])
=
\int_0^T \|\dot\sigma(t)\|_F\,dt
\le
\int_0^T \|\dot\pi(t)\|_F\,dt
=
\Var(\pi;[0,T]),
\end{equation}
which is \eqref{eq:var-play}.
\end{proof}

In summary, the elastic-plastic history operator $\mathcal{P}$ is causal, stable from the $BV$ norm of the driving history into the $L^\infty$ norm of the stress history, and contractive with respect to total variation. These are precisely the constitutive properties that underlie the approximation analysis that follows.

\section{Optimal $N$-piece step approximation}
\label{sec:optimal}

The preceding analysis suggests representing hereditary histories by right-continuous step surrogates and then decoding them through the same constitutive law. By an \emph{encoder} we understand a map
\begin{equation} \label{eq:encoder}
    f_N : W^{1,1}([0,T];F) \to \mathcal{Z}_N,
\end{equation}
where $\mathcal{Z}_N$ denotes the class of right-continuous step functions with at most $N$ constant pieces. More precisely,
\begin{equation} \label{eq:ZN}
    \mathcal{Z}_N
    :=
    \left\{
        z:[0,T]\to F :
        \begin{array}{l}
        \exists\ \Pi_N=\{0=t_0<\cdots<t_M=T\},\ M\le N,\ (z_k)_{k=0}^M\subset F,\\
        z(t)=z_{k-1}\ \text{for }t\in [t_{k-1},t_k),\ k=1,\dots,M,\ \text{and } z(T)=z_M
        \end{array}
    \right\}.
\end{equation}
The set $\mathcal Z_N$ is not a linear space, since the sum of two step histories with at most $N$ pieces may have more than $N$ pieces. Nevertheless, $\mathcal Z_N\hookrightarrow BV([0,T];F)$.

For $z\in\mathcal Z_N$ we write
\begin{equation} \label{eq:discrete-decoder}
    g_N^\star(z):=\mathcal P_{BV}(z),
\end{equation}
which is exactly the closest-point catching-up recursion of Section~\ref{sec:hereditary} applied to the step input $z$. The corresponding approximate stress history generated by the encoded history $f_N(\pi)$ is
\begin{equation} \label{eq:encoder-decoder}
    \mathcal{P}_N(\pi):=g_N^\star(f_N(\pi)).
\end{equation}
Approximations of the form \eqref{eq:encoder-decoder} are \emph{encoder/decoder} approximations in the sense of information-based complexity \cite{NovakWozniakowski2008, NovakWozniakowski2010, NovakWozniakowski2012}. The encoder stores a finite history surrogate; the decoder is the hereditary material law itself.

It is useful to distinguish two kinds of finite code. An \emph{input code} approximates the driving history $\pi$ and is useful when one wants a reduced representation of the prescribed strain history. A \emph{constitutive code} is allowed to be material-law aware and may encode a surrogate of the stress history $\mathcal P(\pi)$, while still being decoded by the same step-input hereditary law. The minimax input error below addresses the first question. The minimax constitutive error addresses the second.

\subsection{Admissible surrogates and error functionals}

Given a corresponding encoder of the form (\ref{eq:encoder}), we distinguish the input approximation error
\begin{equation}
    A(f_N)
    :=
    \sup_{\pi\in  W^{1,1}([0,T];F)\setminus\{0\}}
    \frac
    {
        \|\pi-f_N(\pi)\|_{L^\infty([0,T];F)}
    }
    {
        \|\pi\|_{BV,0}
    } ,
\end{equation}
from the constitutive error
\begin{equation}
    E(f_N)
    :=
    \sup_{\pi\in  W^{1,1}([0,T];F)\setminus\{0\}}
    \frac
    {
        \|\mathcal{P}(\pi)-g_N^\star(f_N(\pi))\|_{L^\infty([0,T];F)}
    }
    {
        \|\pi\|_{BV,0}
    } .
\end{equation}

The best achievable values are, respectively,
\begin{equation}
    A_{N}
    :=
    \inf_{f_N: W^{1,1}([0,T];F)\to \mathcal{Z}_N} A(f_N),
    \quad
    E_{N}
    :=
    \inf_{f_N: W^{1,1}([0,T];F)\to \mathcal{Z}_N} E(f_N).
\end{equation}

Our main objective is to determine the sharp values of $A_N$ and $E_N$, and to identify what kind of encoder attains them.

\subsection{Input approximation}

A lower bound is already encoded in the simplest possible family of inputs, namely ramps contained entirely in the elastic domain.

\begin{lemma}[Sharp lower bound on a segment] \label{lem:quant}
Let $v\in F$ with $\|v\|_F=1$ and let
\begin{equation}
    \pi_\alpha(t):=\frac{\alpha t}{T}\,v, \quad t\in[0,T].
\end{equation}
Then for every $z\in\mathcal{Z}_N$,
\begin{equation} \label{eq:ramp-lower}
    \|\pi_\alpha-z\|_{L^\infty([0,T];F)}
    \geq
    \frac{\alpha}{2N}.
\end{equation}
\end{lemma}

\begin{proof}
Fix $z\in \mathcal{Z}_N$. Let $P_v$ be the orthogonal projection onto $\mathrm{span}\{v\}$. Since $P_v$ is $1$-Lipschitz and $P_v\pi_\alpha=\pi_\alpha$, we have
\begin{equation}
\|\pi_\alpha-z\|_{L^\infty([0,T];F)}
\ge
\|\pi_\alpha-P_v z\|_{L^\infty([0,T];F)}.
\end{equation}
Thus we may assume $z(t)=s(t)v$, where $s$ is a scalar step function with at most $M\le N$ pieces. Set $r(t):=\alpha t/T$. On each interval $[t_{k-1},t_k)$ where $s\equiv c_k$, the oscillation of $r$ is
\begin{equation}
\Delta_k = r(t_k)-r(t_{k-1}) = \frac{\alpha}{T}(t_k-t_{k-1}) .
\end{equation}
Therefore,
\begin{equation}
\sup_{t\in[t_{k-1},t_k)} \|r(t)-c_k\| \ge \frac{\Delta_k}{2} ,
\end{equation}
and
\begin{equation}
\|r-s\|_{L^\infty(0,T)} \ge \max_k \frac{\Delta_k}{2}.
\end{equation}
Since $\sum_{k=1}^M \Delta_k = \alpha$, not all $\Delta_k$ can be smaller than $\alpha/M$. Hence, there exists $k$ such that $\Delta_k \ge {\alpha}/{M} \ge {\alpha}/{N}$, and
\begin{equation}
\|r-s\|_{L^\infty(0,T)} \ge \frac{\alpha}{2N} ,
\end{equation}
which proves the claim.
\end{proof}

The matching upper bound is obtained directly by sampling at equal increments of cumulative variation. This construction is minimax-optimal, although it is not asserted to be the pointwise Chebyshev-best step approximation of every individual vector-valued path.

\subsubsection{Sampled equal-variation encoder}

Let $\pi\in W^{1,1}([0,T];F)$ and set $V:=\Var(\pi;[0,T])$. If $V=0$, define $f_N^\star(\pi)\equiv \pi(0)$. Assume henceforth that $V>0$. Set
\begin{equation}
    \alpha(t):=\Var(\pi;[0,t]), \quad t\in[0,T].
\end{equation}
Because $\pi\in W^{1,1}([0,T];F)$ is continuous, $\alpha$ is continuous as well. For $k=0,\dots,N$ let
\begin{equation} \label{eq:eqvar-times}
    t_k:=\min\Bigl\{t\in[0,T]:\alpha(t)=\frac{kV}{N}\Bigr\},
\end{equation}
and choose
\begin{equation} \label{eq:tauk-cont}
    \tau_k\in[t_{k-1},t_k]
    \quad\text{such that}\quad
    \alpha(\tau_k)=\frac{(2k-1)V}{2N},
    \quad
    k=1,\dots,N.
\end{equation}
The \emph{sampled equal-variation encoder} is then defined as
\begin{equation} \label{eq:variation-midpoint}
    f_N^\star(\pi)(t):=\pi(\tau_k),
    \;\, t\in [t_{k-1},t_k), \;\, k=1,\dots,N,
    \quad \text{and} \quad
    f_N^\star(\pi)(T):=\pi(T) .
\end{equation}
By construction, $f_N^\star(\pi)\in \mathcal{Z}_N$.

\begin{lemma}[Uniform approximation]
\label{lem:upper-cont}
For every $\pi\in W^{1,1}([0,T];F)$,
\begin{equation} \label{eq:pointwise-upper-cont}
    \|\pi-f_N^\star(\pi)\|_{L^\infty([0,T];F)}
    \le
    \frac{1}{2N}\Var(\pi;[0,T]).
\end{equation}
\end{lemma}

\begin{proof}
If $t=T$, the estimate is trivial because $f_N^\star(\pi)(T)=\pi(T)$. Fix $k\in\{1,\dots,N\}$ and $t\in [t_{k-1},t_k)$. By continuity and the definition of $t_k$ and $\tau_k$,
\begin{equation}
\Var(\pi;[t_{k-1},\tau_k])=\Var(\pi;[\tau_k,t_k])=\frac{V}{2N} , 
\end{equation}
where we use the property $\operatorname{Var}(\pi;[a,b]) = \alpha(b) - \alpha(a)$. 
Therefore,
\begin{equation}
\begin{split}
    \|\pi(t)-\pi(\tau_k)\|_F
    & \le
    \Var(\pi;[\min\{t,\tau_k\},\max\{t,\tau_k\}])
    \\ & \le
    \max\left\{\Var(\pi;[t_{k-1},\tau_k]),\Var(\pi;[\tau_k,t_k])\right\}
    =
    \frac{V}{2N}.
\end{split}
\end{equation}
Taking the supremum over $t$ proves the claim.
\end{proof}

\begin{theorem}[Sharp minimax input approximation]
\label{thm:sharp-cont}
For every $N\in\mathbb{N}$,
\begin{equation}
    A_{N}=\frac{1}{2N}.
\end{equation}
Moreover, the value is attained by the sampled equal-variation encoder $f_N^\star$ defined by \eqref{eq:variation-midpoint}. In particular,
\begin{equation} \label{eq:sharp-upper-cont}
    \|\pi-f_N^\star(\pi)\|_{L^\infty([0,T];F)}
    \le
    \frac{1}{2N}\,\|\pi\|_{BV,0},
    \quad
    \forall \pi\in  W^{1,1}([0,T];F).
\end{equation}
\end{theorem}

\begin{proof}
Lemma~\ref{lem:upper-cont} gives
\begin{equation}
    \|\pi-f_N^\star(\pi)\|_{L^\infty}
    \le
    \frac{1}{2N}\Var(\pi;[0,T])
    \le
    \frac{1}{2N}\|\pi\|_{BV,0}.
\end{equation}
Hence $A_{N}\le (2N)^{-1}$. For the reverse inequality, fix $\alpha>0$ and choose a unit vector $v\in F$. Let $f_N: W^{1,1}([0,T];F)\to \mathcal{Z}_N$ be any encoder. Applying Lemma~\ref{lem:quant} to $z=f_N(\pi_\alpha)$ gives
\begin{equation}
  \|\pi_\alpha-f_N(\pi_\alpha)\|_{L^\infty}
  \ge
  \frac{\alpha}{2N}.
\end{equation}
Because $\|\pi_\alpha\|_{BV,0}=\alpha$, it follows that $A(f_N)\ge (2N)^{-1}$. Taking the infimum over $f_N$ proves the theorem.
\end{proof}

\subsection{Constitutive consequences in the vector case}

The sharp input theorem does not by itself imply a sharp constitutive theorem. The reason is structural: in the vector case the play operator is stable from $BV$ to $L^\infty$, but it is not $L^\infty$-nonexpansive, as shown in Remark~\ref{rem:failure-Linfty}. Thus an $L^\infty$-optimal approximation of the input cannot simply be pushed through the decoder.

The constitutive minimax quantity $E_N$, however, does not require the encoder to be a pointwise approximation of $\pi$. It only requires the encoded object to belong to $\mathcal Z_N$ and to be decoded by the material law. This distinction permits the following stress-aware construction: compute the exact stress history $\sigma=\mathcal P(\pi)$, compress $\sigma$ by equal-variation sampling, and pass that admissible stress-valued step history through the decoder. Since admissible step histories are fixed points of the play operator, the decoder leaves the code unchanged.

\begin{definition}
Let $\mathcal{Z}_N(C)$ denote the set of right-continuous step histories in $\mathcal Z_N$ whose values lie in $C$:
\begin{equation}
    \mathcal Z_N(C):=\{z\in\mathcal Z_N:\ z(t)\in C\ \text{for all }t\in[0,T]\}.
\end{equation}
For $\sigma\in W^{1,1}([0,T];C)$, let ${S}_N^\star(\sigma)$ denote the sampled equal-variation approximation defined by \eqref{eq:eqvar-times}--\eqref{eq:variation-midpoint}, with $\pi$ replaced by $\sigma$.
\end{definition}

By Lemma~\ref{lem:upper-cont}, applied to the $C$-valued path $\sigma$, one has
\begin{equation} \label{eq:stress-eqvar-estimate}
    \|\sigma-{S}_N^\star(\sigma)\|_{L^\infty([0,T];F)}
    \le
    \frac{1}{2N}\,\Var(\sigma;[0,T]),
    \quad
    {S}_N^\star(\sigma)\in\mathcal Z_N(C).
\end{equation}
The fact that the samples remain in $C$ uses only that $\sigma(t)\in C$ for all $t$.

\begin{lemma}[Fixed points of the discrete decoder]
\label{lem:decoder-fixed-points}
For every $z\in\mathcal Z_N(C)$,
\begin{equation}
    g_N^\star(z)=z.
\end{equation}
Consequently, $g_N^\star(\mathcal Z_N)=\mathcal Z_N(C)$ and $g_N^\star\circ g_N^\star=g_N^\star$ on $\mathcal Z_N$.
\end{lemma}

\begin{proof}
If $z\in\mathcal Z_N(C)$, Proposition~\ref{rel-proy} shows that the step-input constitutive response to $z$ is exactly $z$. Hence $g_N^\star(z)=z$. Since every decoded history takes values in $C$, $g_N^\star(\mathcal Z_N)\subset\mathcal Z_N(C)$. The reverse inclusion follows from the fixed-point property just proved. The projector identity follows immediately.
\end{proof}

\begin{theorem}[Sharp vector constitutive minimax theorem]
\label{thm:sharp-vector-constitutive}
Assume that $0\in\operatorname{int}C$. Then, for every $N\in\mathbb N$,
\begin{equation} \label{eq:sharp-vector-constitutive}
    E_N=\frac{1}{2N}.
\end{equation}
Moreover, the value is attained by the stress-aware encoder
\begin{equation} \label{eq:stress-aware-encoder}
    h_N^\star(\pi):={S}_N^\star(\mathcal P(\pi)).
\end{equation}
Equivalently,
\begin{equation} \label{eq:sharp-vector-upper}
    \|\mathcal P(\pi)-g_N^\star(h_N^\star(\pi))\|_{L^\infty([0,T];F)}
    \le
    \frac{1}{2N}\,\|\pi\|_{BV,0},
    \quad
    \forall\,\pi\in W^{1,1}([0,T];F).
\end{equation}
\end{theorem}

\begin{proof}
Let $\pi\in W^{1,1}([0,T];F)$ and set $\sigma:=\mathcal P(\pi)$. Then $\sigma\in W^{1,1}([0,T];C)$ and, by Theorem~\ref{thm:regularity}(iv),
\begin{equation}
    \Var(\sigma;[0,T])\le \Var(\pi;[0,T]).
\end{equation}
The encoder \eqref{eq:stress-aware-encoder} belongs to $\mathcal Z_N(C)$, hence Lemma~\ref{lem:decoder-fixed-points} gives
\begin{equation}
    g_N^\star(h_N^\star(\pi))
    =h_N^\star(\pi)
    ={S}_N^\star(\sigma).
\end{equation}
Using \eqref{eq:stress-eqvar-estimate},
\begin{equation}
\begin{split}
    \|\mathcal P(\pi)-g_N^\star(h_N^\star(\pi))\|_{L^\infty}
    & =
    \|\sigma-{S}_N^\star(\sigma)\|_{L^\infty}
    \\
    & \le
    \frac{1}{2N}\Var(\sigma;[0,T])
    \le
    \frac{1}{2N}\Var(\pi;[0,T])
    \le
    \frac{1}{2N}\|\pi\|_{BV,0}.
\end{split}
\end{equation}
Thus $E_N\le(2N)^{-1}$.

For the reverse inequality, use the same elastic ramp that proves the input lower bound. Since $0\in\operatorname{int}C$, there is $r>0$ such that $B_F(0,r)\subset C$. Fix $0<\alpha\le r$ and a unit vector $v\in F$, and set $\pi_\alpha(t)=(\alpha t/T)v$. Then $\pi_\alpha(t)\in C$ for all $t$. The history $\sigma=\pi_\alpha$ satisfies $\sigma(0)=P_C(\pi_\alpha(0))$ and $\dot\pi_\alpha-\dot\sigma=0\in N_C(\sigma)$ a.e.; by uniqueness in Theorem~\ref{thm:bv-well-posedness}, $\mathcal P(\pi_\alpha)=\pi_\alpha$. For an arbitrary encoder $f_N:W^{1,1}([0,T];F)\to\mathcal Z_N$, the decoded history $g_N^\star(f_N(\pi_\alpha))$ belongs to $\mathcal Z_N(C)\subset\mathcal Z_N$. Lemma~\ref{lem:quant} therefore implies
\begin{equation}
    \|\mathcal P(\pi_\alpha)-g_N^\star(f_N(\pi_\alpha))\|_{L^\infty}
    =
    \|\pi_\alpha-g_N^\star(f_N(\pi_\alpha))\|_{L^\infty}
    \ge
    \frac{\alpha}{2N}.
\end{equation}
Since $\|\pi_\alpha\|_{BV,0}=\alpha$, $E(f_N)\ge(2N)^{-1}$. Taking the infimum over $f_N$ proves $E_N\ge(2N)^{-1}$, and hence \eqref{eq:sharp-vector-constitutive}. The upper bound shows that $h_N^\star$ attains the value.
\end{proof}

\begin{remark}[Role of the interior assumption] {\rm 
The stress-aware upper bound $E_N\le(2N)^{-1}$ uses only closedness and convexity of $C$. The assumption $0\in\operatorname{int}C$ is used for the lower bound with the normalized $BV,0$ denominator. Without a nondegeneracy assumption of this kind, the statement can fail; for instance, if $C$ is a singleton, then every stress history is constant and $E_N=0$.
} \hfill$\square$
\end{remark}

\begin{remark}[What the input-first estimate gives]
\label{rem:old-BV-route}
{\rm If one insists that the encoder first approximate the input $\pi$ and then estimates the output only through the $BV$ stability of Proposition~\ref{prop:BV-play-stability}, one obtains the true but suboptimal bound
\begin{equation} \label{eq:old-BV-route}
    E(f_N)
    \le
    \sup_{\pi\in W^{1,1}([0,T];F)\setminus\{0\}}
    \frac{\|\pi-f_N(\pi)\|_{BV,0}}{\|\pi\|_{BV,0}}.
\end{equation}
Indeed, apply Proposition~\ref{prop:BV-play-stability} to the pair $(\pi,f_N(\pi))$. The constant encoder $f_N^0(\pi)\equiv\pi(0)$ gives $E(f_N^0)\le1$, but this does not recover the sharp $1/(2N)$ rate. Nor does the sharp $L^\infty$ input encoder help through this route: for the scalar ramp $\pi_\alpha(t)=\alpha t/T$, the midpoint equal-variation encoder satisfies
\begin{equation}
    \|\pi_\alpha-f_N^\star(\pi_\alpha)\|_{L^\infty}=\frac{\alpha}{2N},
\end{equation}
but
\begin{equation}
    \Var(\pi_\alpha-f_N^\star(\pi_\alpha);[0,T])
    =2\alpha-\frac{\alpha}{2N}.
\end{equation}
Thus the $BV$ error remains of order one relative to $\|\pi_\alpha\|_{BV,0}=\alpha$. This is the precise sense in which the input-first $BV$-stability route is stuck with suboptimal estimates. The sharp theorem above bypasses this obstruction by encoding the stress history, not by proving an unavailable vector $L^\infty$ non-expansiveness property.
} \hfill$\square$
\end{remark}

\begin{remark}[Continuity of optimal partitions] 
{\rm The equal-variation encoders $f_N^\star$ and $h_N^\star$ are generally discontinuous as mappings of the input in either the $L^\infty$ or the $BV$ topology, because small perturbations may alter the equal-variation partition abruptly. This lack of continuity reflects the fact that the optimal partition is itself part of the code. By contrast, the completed-graph constitutive map $\mathcal P_{BV}$ remains robust in the $BV$-to-$L^\infty$ sense of Proposition~\ref{prop:BV-play-stability}.
} \hfill$\square$
\end{remark}

\subsection{Constitutive consequences in the scalar case}

The vector theorem above settles the stress-valued minimax problem when stress-aware encoders are admissible. The scalar case records a different and useful fact: for the complementary variable, a purely input-based encoder is already sharp. The reason is that in one space dimension the complementary-variable map is $L^\infty$-nonexpansive with respect to the driving history.

Assume that $\dim F = 1$. We identify $F$ with $\mathbb{R}$ and write $C=[\alpha,\beta]$, $\alpha<\beta$. For $\pi\in W^{1,1}(0,T)$, define the scalar complementary-variable operator
\begin{equation}
W(\pi):=\pi-\mathcal P(\pi).
\end{equation}
More generally, for $z\in BV(0,T)$, define
\begin{equation}
W_{BV}(z):=z-\mathcal P_{BV}(z).
\end{equation}
In particular, if $z\in \mathcal{Z}_N$, then by $BV$-compatibility,
\begin{equation}
W_{BV}(z)=z-\mathcal P_{BV}(z)=z-g_N^\star(z),
\end{equation}
by the step-input compatibility stated above. Thus, given an encoder $f_N:W^{1,1}(0,T) \to \mathcal{Z}_N$, the decoded approximation of the complementary variable is
\begin{equation}
W_N(\pi):=W_{BV}(f_N(\pi))
=f_N(\pi)-g_N^\star(f_N(\pi)).
\end{equation}
We measure the constitutive error in the complementary variable by
\begin{equation}
{D}(f_N):=
\sup_{\pi\in W^{1,1}(0,T)\setminus\{0\}}
\frac{\|W(\pi)-W_N(\pi)\|_{L^\infty(0,T)}}{\|\pi\|_{BV,0}},
\end{equation}
and define
\begin{equation}
{D}_N:=
\inf_{f_N:W^{1,1}(0,T)\to \mathcal{Z}_N} {D}(f_N).
\end{equation}

\begin{lemma}[$L^\infty$-nonexpansiveness of the scalar complementary variable]
\label{lem:scalar-W-nonexpansive-BV}
For every $\pi\in W^{1,1}(0,T)$ and every $z\in BV(0,T)$,
\begin{equation} \label{eq:nonexp-scalar}
\|W(\pi)-W_{BV}(z)\|_{L^\infty(0,T)}
\le
\|\pi-z\|_{L^\infty(0,T)}.
\end{equation}
\end{lemma}

\begin{proof}
We prove the slightly stronger statement
\begin{equation}
\label{eq:W-BV-nonexpansive}
\|W_{BV}(z^1)-W_{BV}(z^2)\|_{L^\infty(0,T)}
\le
\|z^1-z^2\|_{L^\infty(0,T)}
\quad
\forall\,z^1,z^2\in BV(0,T)\text{ right-continuous}.
\end{equation}
If $z^1$ and $z^2$ are step inputs, then \eqref{eq:W-BV-nonexpansive} is exactly the scalar estimate in Remark~\ref{rem:scalar-case}, after passing to a common refinement of their partitions. For general right-continuous $BV$ inputs, choose right-continuous step functions $z_m^i$ with $z_m^i\to z^i$ uniformly on $[0,T]$, $i=1,2$. Such approximations exist because $BV$ functions are regulated. The scalar play operator is the uniform closure of the step-input catching-up operators, and the estimate for step inputs makes this closure $1$-Lipschitz in the uniform norm; equivalently, $W_{BV}(z_m^i)\to W_{BV}(z^i)$ uniformly. Passing to the limit in the step estimate gives \eqref{eq:W-BV-nonexpansive}; see \cite{Krejci1996, BrokateSprekels1996}. Finally, apply \eqref{eq:W-BV-nonexpansive} with $z^1=\pi$ and $z^2=z$, using $W_{BV}(\pi)=W(\pi)$ for $\pi\in W^{1,1}(0,T)$.
\end{proof}

\begin{proposition}[Constitutive error controlled by $L^\infty$ input error]
\label{prop:constitutive-scalar}
For every encoder $f_N:W^{1,1}(0,T)\to \mathcal{Z}_N$,
\begin{equation}
{D}(f_N)\le A(f_N).
\end{equation}
In particular, for the sampled equal-variation encoder $f_N^\star$,
\begin{equation}
{D}(f_N^\star)\le \frac{1}{2N},
\end{equation}
and therefore
\begin{equation}
\label{eq:upper-bound-scalar}
{D}_N\le \frac{1}{2N}.
\end{equation}
\end{proposition}

\begin{proof}
Fix $\pi\in W^{1,1}(0,T)$ and set $z:=f_N(\pi)\in \mathcal{Z}_N$. By
Lemma~\ref{lem:scalar-W-nonexpansive-BV},
\begin{equation}
\begin{split}
\|W(\pi)-W_N(\pi)\|_{L^\infty(0,T)}
& =
\|W(\pi)-W_{BV}(z)\|_{L^\infty(0,T)}
\\ & \le
\|\pi-z\|_{L^\infty(0,T)}
=
\|\pi-f_N(\pi)\|_{L^\infty(0,T)}.
\end{split}
\end{equation}
Dividing by $\|\pi\|_{BV,0}$ and taking the supremum over $\pi\neq 0$ gives
\begin{equation}
{D}(f_N)\le
\sup_{\pi\in W^{1,1}(0,T)\setminus\{0\}}
\frac{\|\pi-f_N(\pi)\|_{L^\infty(0,T)}}{\|\pi\|_{BV,0}}
=
A(f_N).
\end{equation}
Applying this estimate to $f_N^\star$ and using Theorem~\ref{thm:sharp-cont} yields
\begin{equation}
{D}(f_N^\star)\le A(f_N^\star)=\frac{1}{2N}.
\end{equation}
Taking the infimum over all encoders proves (\ref{eq:upper-bound-scalar}).
\end{proof}

\begin{theorem}[Sharp scalar constitutive minimax theorem]
\label{thm:sharp-scalar-constitutive}
For every $N\in\mathbb{N}$,
\begin{equation}
\label{eq:error-scalar}
{D}_N=\frac{1}{2N}.
\end{equation}
Moreover, the value is attained by the sampled equal-variation encoder $f_N^\star$.
\end{theorem}

\begin{proof}
The upper bound (\ref{eq:upper-bound-scalar}) has already been proved in Proposition~\ref{prop:constitutive-scalar}. It remains to prove the reverse inequality. Let $f_N:W^{1,1}(0,T)\to \mathcal{Z}_N$ be any encoder, and for $\lambda>0$ consider the monotone loading
\begin{equation}
\pi_\lambda(t):=\beta+\frac{\lambda t}{T},
\quad t\in[0,T].
\end{equation}
Since $\pi_\lambda(0)=\beta\in C$ and $\pi_\lambda$ is nondecreasing, the scalar play operator remains stuck at the right endpoint,
\begin{equation}
\mathcal P(\pi_\lambda)(t)\equiv \beta
\quad \forall\, t\in[0,T].
\end{equation}
Therefore, the exact complementary output is the ramp
\begin{equation}
W(\pi_\lambda)(t)=\pi_\lambda(t)-\mathcal P(\pi_\lambda)(t)=\frac{\lambda t}{T}.
\end{equation}
Set $z_\lambda:=f_N(\pi_\lambda)\in \mathcal{Z}_N$. By definition,
\begin{equation}
W_N(\pi_\lambda)=W_{BV}(z_\lambda)=z_\lambda-g_N^\star(z_\lambda).
\end{equation}
Since both $z_\lambda$ and $g_N^\star(z_\lambda)$ are right-continuous step functions on the same partition, $W_N(\pi_\lambda)$ is itself a right-continuous step function with at most $N$ constant pieces, i.e.,
\begin{equation}
W_N(\pi_\lambda)\in \mathcal{Z}_N.
\end{equation}
Applying Lemma~\ref{lem:quant} to the ramp $t\mapsto \lambda t/T$ therefore gives
\begin{equation}
\|W(\pi_\lambda)-W_N(\pi_\lambda)\|_{L^\infty(0,T)}
\ge
\frac{\lambda}{2N}.
\end{equation}
In addition,
\begin{equation}
\|\pi_\lambda\|_{BV,0}
=
|\pi_\lambda(0)|+\Var(\pi_\lambda;[0,T])
=
|\beta|+\lambda.
\end{equation}
Hence,
\begin{equation}
{D}(f_N)
\ge
\frac{\|W(\pi_\lambda)-W_N(\pi_\lambda)\|_{L^\infty(0,T)}}{\|\pi_\lambda\|_{BV,0}}
\ge
\frac{\lambda}{2N(|\beta|+\lambda)}.
\end{equation}
This lower bound holds for every $\lambda>0$. Letting $\lambda\to\infty$ yields
\begin{equation}
{D}(f_N)\ge \frac{1}{2N}.
\end{equation}
Since $f_N$ is arbitrary,
\begin{equation}
{D}_N\ge \frac{1}{2N}.
\end{equation}
Combining this with Proposition~\ref{prop:constitutive-scalar} proves (\ref{eq:error-scalar}). The same proposition shows that $f_N^\star$ attains the upper bound, hence the minimum is achieved by $f_N^\star$.
\end{proof}

\section{Numerical illustrations} \label{sec:numerical}

We present numerical illustrations that distinguish two different uses of equal-variation sampling. The first is the input-first use, in which the prescribed driving history is compressed and then decoded through the material law. The second is the stress-aware use appearing in Theorem~\ref{thm:sharp-vector-constitutive}, in which the exact stress history is first computed, sampled by equal stress variation, and then passed through the same discrete hereditary decoder. The latter construction is a constitutive code rather than an input approximation: because the sampled stress history takes values in the elastic domain, Lemma~\ref{lem:decoder-fixed-points} implies that the decoder leaves it fixed.

All errors in the figures below are measured in the same stress-space norm used in the analysis. The smooth reference histories are evaluated on a fine background grid and the step encoders are then compared with the reference histories on that grid. The purpose of the computations is not to replace the minimax proof, but to make visible the geometric distinction between sampling the input path and sampling the stress path.

\subsection{Mises solid under proportional loading}

We consider an isotropic three-dimensional von Mises solid with
\begin{equation} \label{kWMoYt}
    C:=\{\sigma\in F:\ |\operatorname{dev}\sigma|\le \sqrt{2}\,\tau_0\},
\end{equation}
and restrict the evolution to the fixed deviatoric direction
\begin{equation}
    D:=\frac{1}{\sqrt{2}}\operatorname{diag}(1,-1,0), 
    \quad 
    \operatorname{tr}D=0, 
    \quad 
    |D|=1 .
\end{equation}
We seek histories of the form
\begin{equation}
    \pi(t)=r(t)D, \quad \sigma(t)=s(t)D,
\end{equation}
for scalar amplitudes $r(t)$ and $s(t)$. The yield condition then becomes
\begin{equation}
    |s(t)|\le s_{\mathrm Y}, \quad s_{\mathrm Y}:=\sqrt{2}\,\tau_0,
\end{equation}
and the constitutive law reduces to the scalar stop operator
\begin{equation} \label{eq:stop_update}
    s_0=\operatorname{proj}_{[-s_{\mathrm Y},s_{\mathrm Y}]}(r(0)),
    \quad
    s_k=\operatorname{proj}_{[-s_{\mathrm Y},s_{\mathrm Y}]}
    \bigl(s_{k-1}+r(t_k)-r(t_{k-1})\bigr),
    \quad k=1,\dots,M,
\end{equation}
on a time grid $0=t_0<t_1<\cdots<t_M=T$; see \cite{OrtizHerrera1981, Krejci1996, BrokateSprekels1996, MielkeRoubicek2015}.

For definiteness, we take $T=10$, $\tau_0=0.8$, and
\begin{equation} \label{eq:example-input}
    r(t)=
    1.6\,\sin\!\Bigl(\frac{4\pi}{T}\,\vartheta(t)\Bigr)
    +
    0.25\,\sin\!\Bigl(\frac{2\pi}{T}\,\vartheta(t)+0.8\Bigr),
    \quad
    \vartheta(t):=t+1.2\sin\!\Bigl(\frac{2\pi t}{T}\Bigr).
\end{equation}
This input is smooth, but its variation is strongly non-uniform in time owing to concentration near turning regions. Consequently, it is a natural candidate for comparing equal-variation sampling against equal-time sampling and, separately, for comparing input-first and stress-aware encodings.

We first compare two right-continuous input encoders:
\begin{enumerate}
    \item[(i)] the equal-variation input encoder $f_N^\star$;
    \item[(ii)] the equal-time input encoder $\widetilde f_N$ defined by the uniform partition $t_k=kT/N$, the right-continuous samples $\widetilde f_N(\pi)(t)=\pi(t_{k-1})$ on $[t_{k-1},t_k)$, and $\widetilde f_N(\pi)(T)=\pi(T)$.
\end{enumerate}
The corresponding input-first stress histories are computed by the discrete decoder,
\begin{equation}
    \sigma_N^\star:=g_N^\star(f_N^\star(\pi)),
    \quad
    \widetilde{\sigma}_N:=g_N^\star(\widetilde f_N(\pi)).
\end{equation}
We consider both the input approximation errors
\begin{equation} \label{eq:input-error}
    E_N^{\mathrm{in}}
    :=
    \|\pi-f_N^\star(\pi)\|_{L^\infty(0,T;F)},
    \quad
    \widetilde E_N^{\mathrm{in}}
    :=
    \|\pi-\widetilde f_N(\pi)\|_{L^\infty(0,T;F)},
\end{equation}
and the input-first stress errors
\begin{equation} \label{eq:stress-error}
    E_N^{\mathrm{out}}
    :=
    \|\sigma-\sigma_N^\star\|_{L^\infty(0,T;F)},
    \quad
    \widetilde E_N^{\mathrm{out}}
    :=
    \|\sigma-\widetilde \sigma_N\|_{L^\infty(0,T;F)}.
\end{equation}
Since all histories are supported on the single mode $D$, these are simply the corresponding scalar $L^\infty$-errors in the scalar amplitudes $r(t)$ and $s(t)$.

\begin{figure}[htbp]
  \centering
  \safeincludegraphics[width=0.86\linewidth]{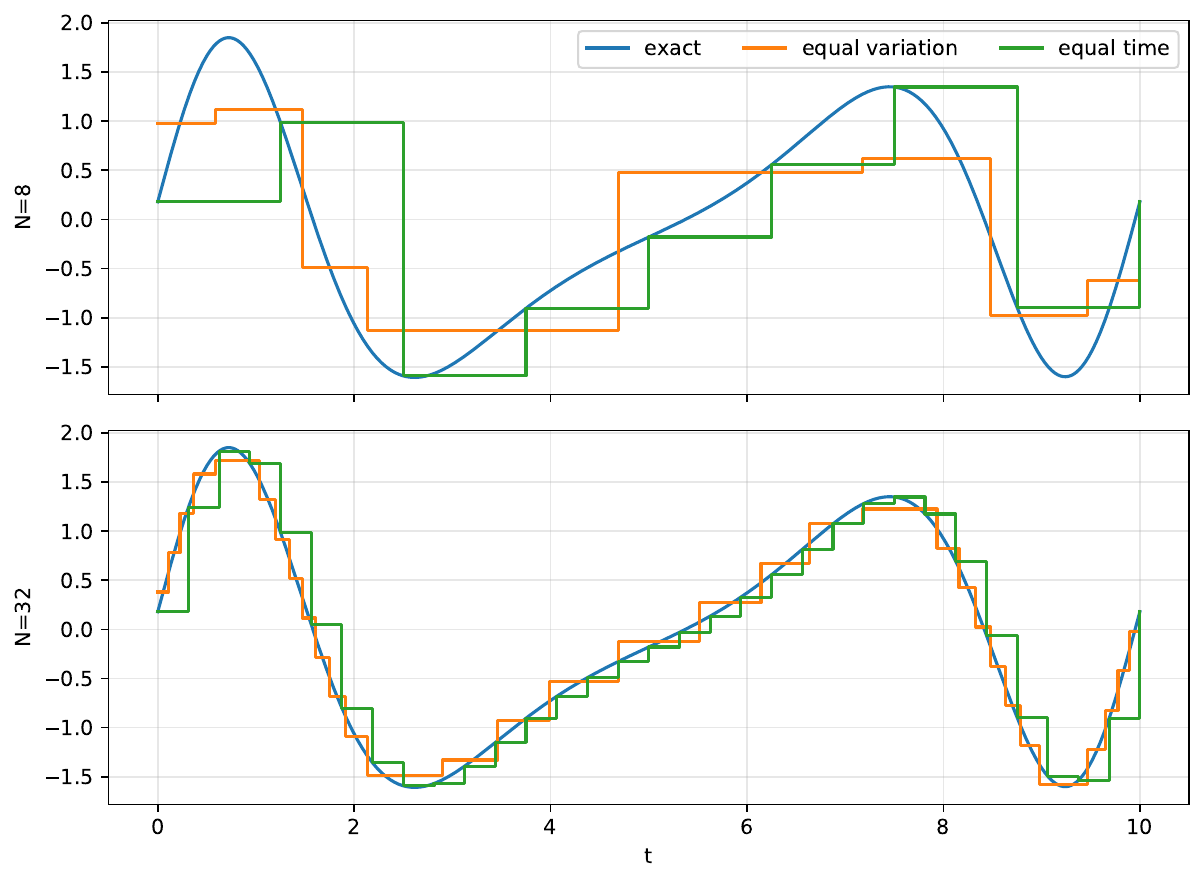}
  \caption{Input amplitude $r(t)$ together with equal-variation and equal-time input encoders for representative values of $N$.} \label{fig:opt-input-pairs}
\end{figure}

\begin{figure}[htbp]
  \centering
  \safeincludegraphics[width=0.86\linewidth]{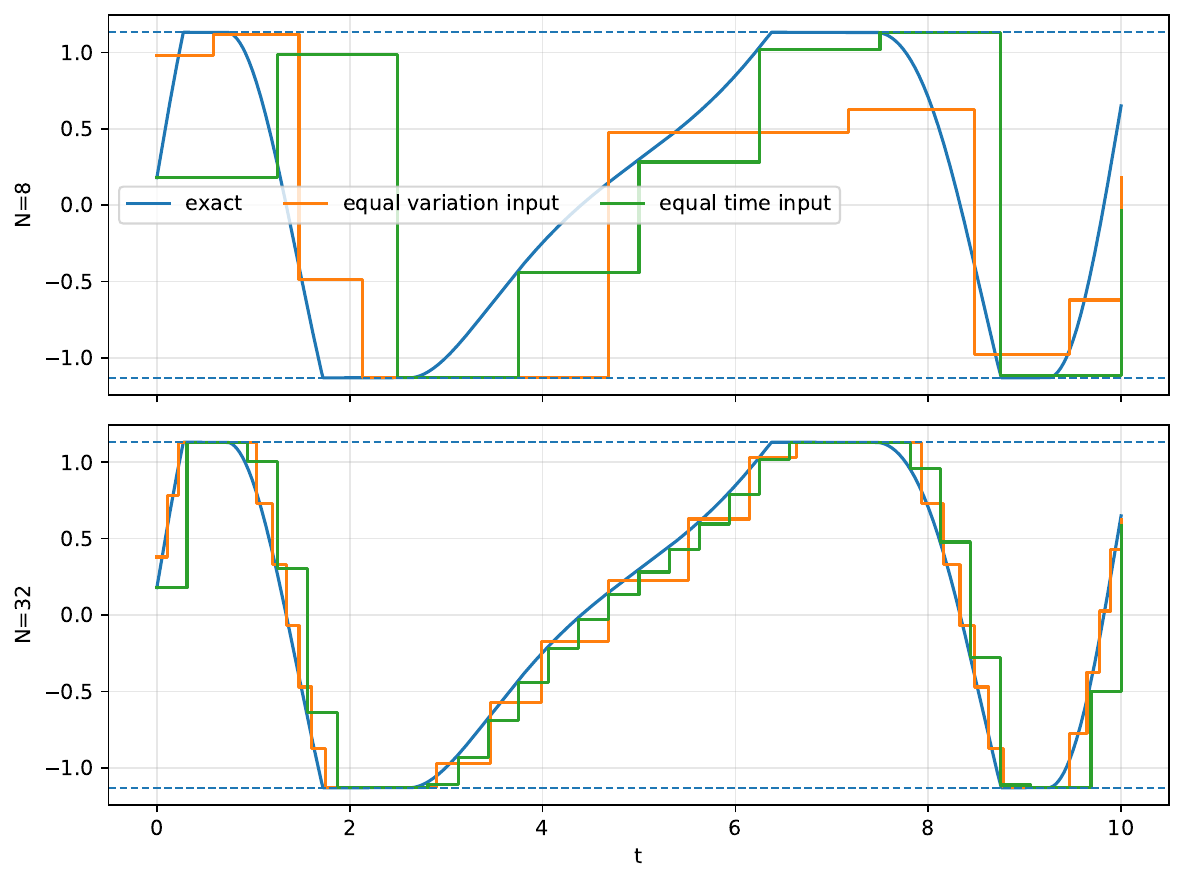}
  \caption{Input-first stress outputs. The dashed lines indicate the admissible bounds $\pm s_{\mathrm Y}$.} \label{fig:opt-stress-pairs}
\end{figure}

The equal-variation input encoder distributes its $N$ values according to the cumulative variation of the input, rather than according to elapsed time. For the input signal \eqref{eq:example-input}, where much of the variation is concentrated near a small number of turning regions, it spends more of its degrees of freedom where the input changes most. By contrast, the equal-time encoder allocates the same number of intervals to slowly varying and rapidly varying parts of the input signal. Figures~\ref{fig:opt-input-pairs} and \ref{fig:opt-stress-pairs} show this mechanism at the input level and after input-first decoding through the scalar stop law.

\begin{figure}[htbp]
  \centering
  \safeincludegraphics[width=0.78\linewidth]{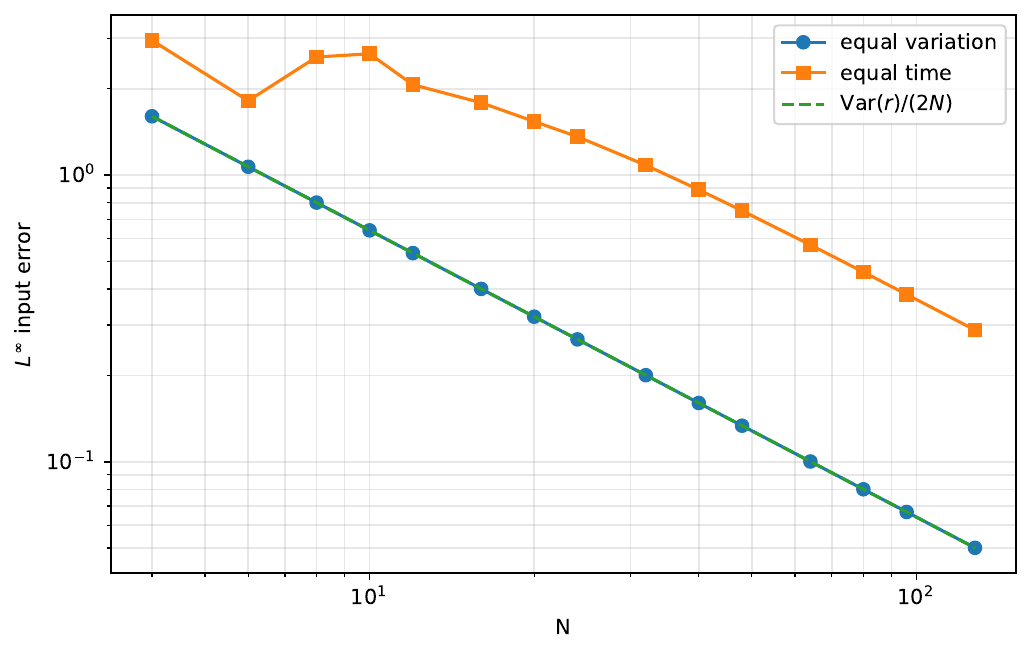}
  \caption{Input approximation error versus $N$.}
  \label{fig:opt-input-error}
\end{figure}

\begin{figure}[htbp]
  \centering
  \safeincludegraphics[width=0.78\linewidth]{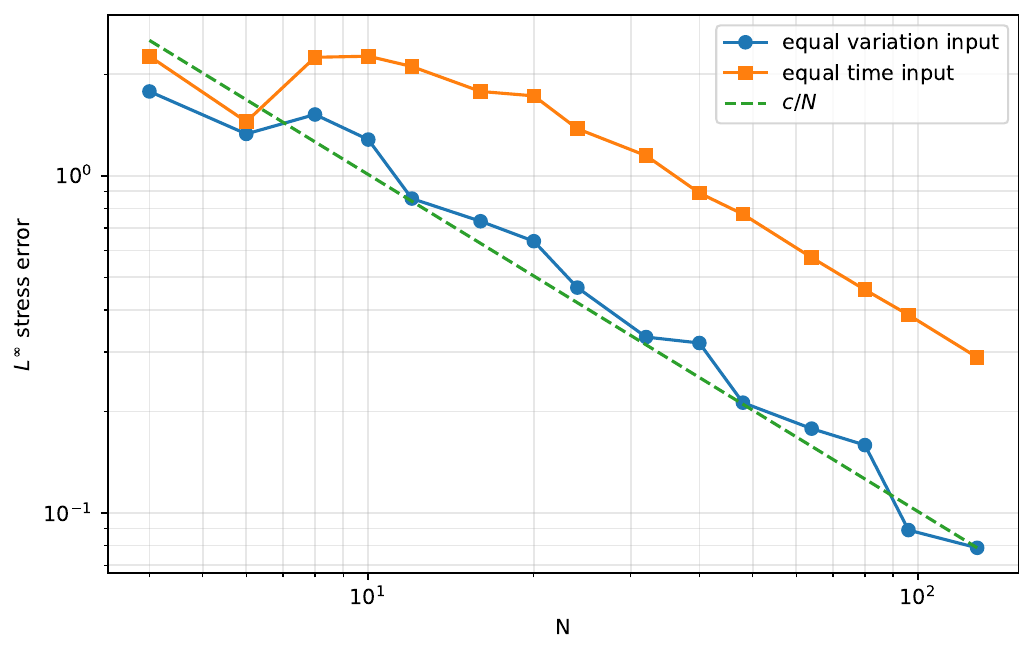}
  \caption{Input-first stress approximation error versus $N$. The dashed line is a $c/N$ reference slope.}
  \label{fig:opt-stress-error}
\end{figure}

Figures~\ref{fig:opt-input-error} and \ref{fig:opt-stress-error} show the input and input-first output errors as a function of the rank of the encoder. The input error follows the $1/N$ scaling predicted by Theorem~\ref{thm:sharp-cont}. The stress output error also exhibits an empirical $1/N$ trend over the displayed range, with the equal-variation input encoder outperforming the equal-time input encoder in this one-dimensional example. These curves, however, are still generated by input-first encoders.

We now display the stress-aware construction from Theorem~\ref{thm:sharp-vector-constitutive}. Let $\sigma=\mathcal P(\pi)$ and define
\begin{equation} \label{eq:prop-stress-aware-code}
    h_N^\star(\pi):={S}_N^\star(\sigma),
    \qquad
    \sigma_N^{\mathrm{sa}}
    :=g_N^\star(h_N^\star(\pi)).
\end{equation}
Because $h_N^\star(\pi)\in\mathcal Z_N(C)$, Lemma~\ref{lem:decoder-fixed-points} gives
\begin{equation}
    \sigma_N^{\mathrm{sa}}=h_N^\star(\pi)={S}_N^\star(\sigma).
\end{equation}
Thus the stress-aware approximation is obtained by sampling the stress path itself by equal increments of stress variation. Its error is
\begin{equation} \label{eq:prop-stress-aware-error}
    E_N^{\mathrm{sa}}
    :=\|\sigma-\sigma_N^{\mathrm{sa}}\|_{L^\infty(0,T;F)}
    \le
    \frac{1}{2N}\,\Var(\sigma;[0,T]),
\end{equation}
which is the computable version of the upper bound used in Theorem~\ref{thm:sharp-vector-constitutive} for this particular history.

\begin{figure}[htbp]
  \centering
  \safeincludegraphics[width=0.86\linewidth]{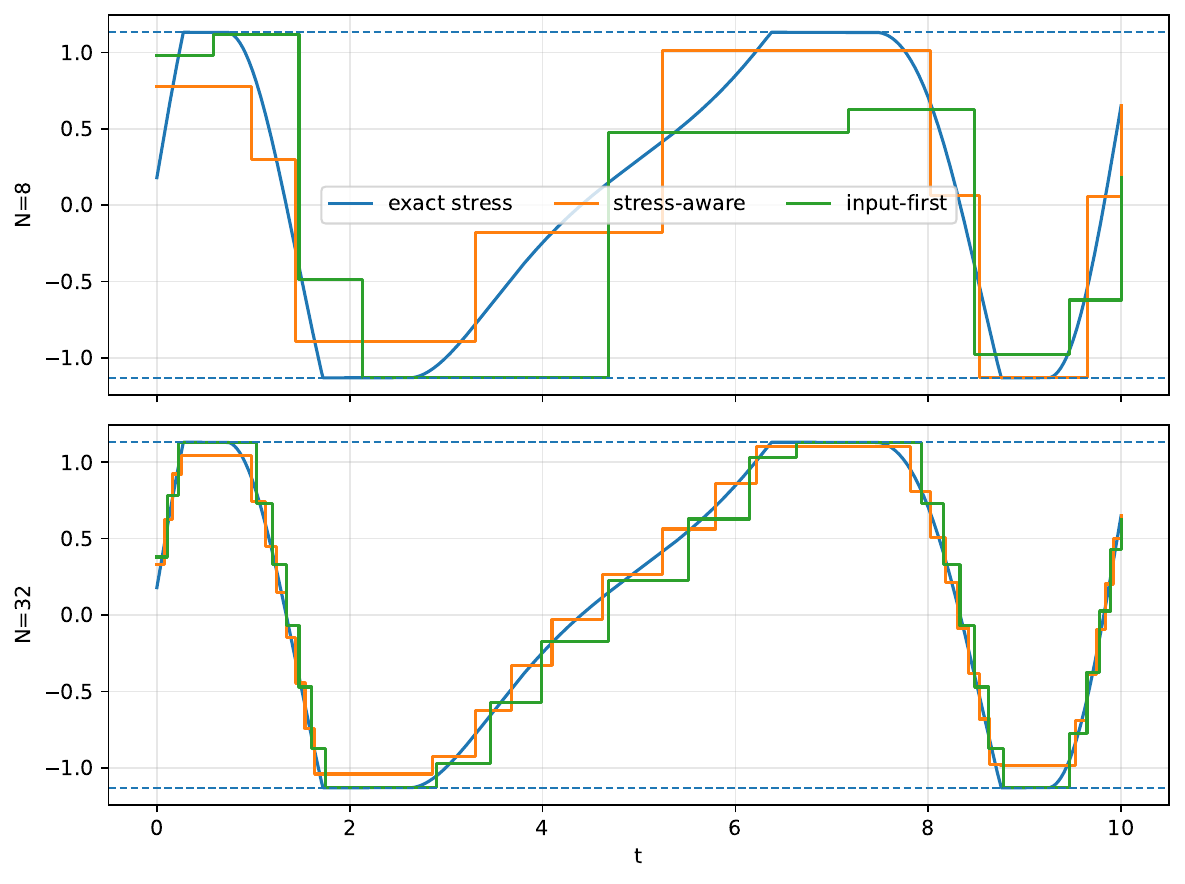}
  \caption{Proportional loading: exact stress, input-first decoded stress, and stress-aware decoded stress for representative values of $N$. The stress-aware decoder leaves the encoded stress samples fixed because they lie in the interval $[-s_{\mathrm Y},s_{\mathrm Y}]$.}
  \label{fig:prop-stressaware-pairs}
\end{figure}

\begin{figure}[htbp]
  \centering
  \safeincludegraphics[width=0.78\linewidth]{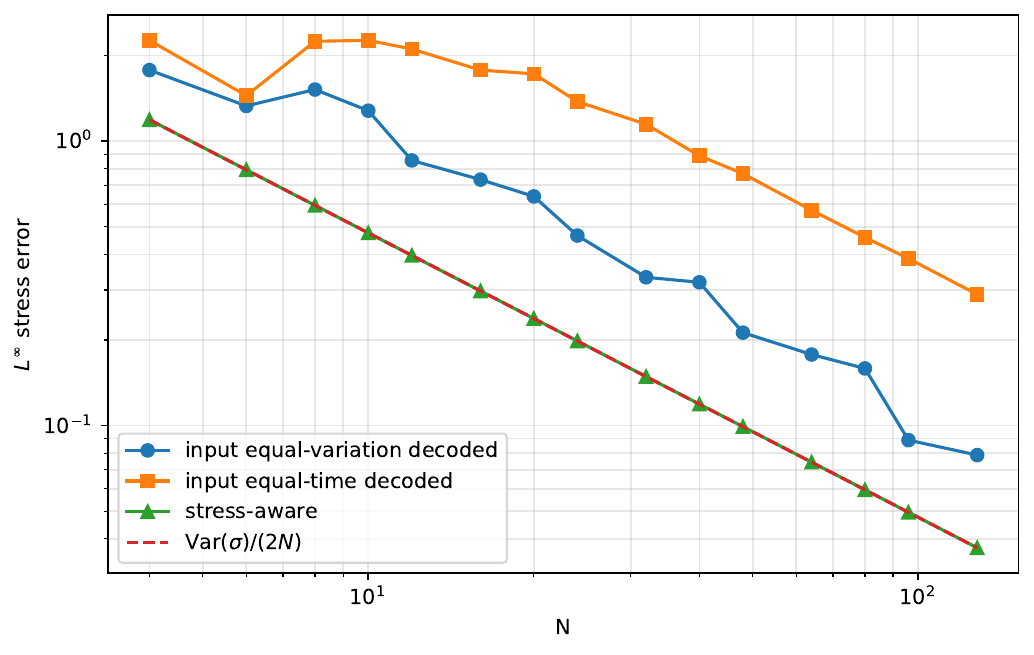}
  \caption{Proportional loading: comparison of input-first stress errors and stress-aware stress error. The dashed curve is the bound $\Var(\sigma;[0,T])/(2N)$.}
  \label{fig:prop-stressaware-error}
\end{figure}

Figure~\ref{fig:prop-stressaware-pairs} shows the new stress-aware approximation together with the input-first equal-variation stress output. The stress-aware jumps are placed according to the variation of $s(t)$, not according to the variation of $r(t)$. In particular, elastic intervals in which the input changes but the stress is nearly saturated are treated differently by the two codes. Figure~\ref{fig:prop-stressaware-error} then compares the resulting stress errors. The stress-aware curve follows the sharp stress-variation envelope \eqref{eq:prop-stress-aware-error}, while the input-first curves reflect the additional effect of passing an input surrogate through the hysteretic decoder.

In summary, the proportional example has two complementary messages. Equal-variation sampling of the input is the optimal input code and is useful when only the loading history is to be compressed. Equal-variation sampling of the stress is the optimal constitutive code: once the stress path has been computed, the decoder is a fixed point on the stress samples and the sharp stress-variation estimate is attained at the level of the displayed error envelope.

\subsection{Mises solid under nonproportional loading}
\label{sec:mises-nonproportional}

We now consider a genuinely vector-valued loading path for the same isotropic three-dimensional von Mises solid \eqref{kWMoYt}, with the same value $\tau_0=0.8$ and final time $T=10$ used in the preceding proportional-loading example. Let
\begin{equation}
    D_1:=\frac{1}{\sqrt{2}}\operatorname{diag}(1,-1,0),
    \quad
    D_2:=\frac{1}{\sqrt{6}}\operatorname{diag}(1,1,-2),
\end{equation}
so that $D_1:D_2=0$ and $|D_1|=|D_2|=1$. We seek histories of the form
\begin{equation}
    \pi(t)=p_1(t)D_1+p_2(t)D_2,
    \quad
    \sigma(t)=q_1(t)D_1+q_2(t)D_2,
\end{equation}
with
\begin{equation}
    p(t):=(p_1(t),p_2(t))=\rho(t)(\cos\phi(t),\sin\phi(t)),
\end{equation}
where
\begin{align}
    \vartheta(t)&:=t+1.1\sin\!\Bigl(\frac{2\pi t}{T}\Bigr), \\
    \rho(t)&:=1.45
    +0.25\sin\!\Bigl(\frac{2\pi}{T}\,\vartheta(t)+0.3\Bigr)
    +0.08\sin\!\Bigl(\frac{6\pi t}{T}\Bigr), \\
    \phi(t)&:=0.4
    +\frac{2.7\pi}{T}\,\vartheta(t)
    +0.4\sin\!\Bigl(\frac{2\pi t}{T}+0.2\Bigr).
\end{align}
Because $\phi(t)$ is not constant, the loading direction changes continuously in the $(D_1,D_2)$-plane and the evolution is therefore nonproportional.

Restricting the yield condition to $\operatorname{span}\{D_1,D_2\}$ gives the disk
\begin{equation}
    B_{s_{\mathrm Y}}
    :=
    \{q\in\mathbb R^2:\ |q|\le s_{\mathrm Y}\},
    \quad
    s_{\mathrm Y}:=\sqrt{2}\,\tau_0,
\end{equation}
and the discrete constitutive law becomes the two-dimensional closest-point recursion
\begin{equation}
    q_0=\operatorname{proj}_{B_{s_{\mathrm Y}}}(p(0)),
    \quad
    q_k=\operatorname{proj}_{B_{s_{\mathrm Y}}}
    \bigl(q_{k-1}+p(t_k)-p(t_{k-1})\bigr),
    \quad k=1,\dots,M.
    \label{eq:nonprop-update}
\end{equation}
Thus, unlike the proportional example, the problem no longer reduces to a scalar stop operator. The output now depends on changes of direction as well as on changes of amplitude.

As before, we compare the input equal-variation encoder $f_N^\star$ against the equal-time input encoder $\widetilde f_N$ defined on the uniform partition $t_k=kT/N$. The decoded input-first stress histories are
\begin{equation}
    \sigma_N^\star:=g_N^\star(f_N^\star(\pi)),
    \quad
    \widetilde\sigma_N:=g_N^\star(\widetilde f_N(\pi)).
\end{equation}
We measure the input and output errors by
\begin{equation}
    E_N^{\mathrm{in}}:=\|\pi-f_N^\star(\pi)\|_{L^\infty(0,T;F)},
    \quad
    \widetilde E_N^{\mathrm{in}}:=\|\pi-\widetilde f_N(\pi)\|_{L^\infty(0,T;F)},
\end{equation}
\begin{equation}
    E_N^{\mathrm{out}}:=\|\sigma-\sigma_N^\star\|_{L^\infty(0,T;F)},
    \quad
    \widetilde E_N^{\mathrm{out}}:=\|\sigma-\widetilde\sigma_N\|_{L^\infty(0,T;F)},
\end{equation}
and, in order to highlight the genuinely vector effect,
\begin{equation}
    R_N:=\frac{E_N^{\mathrm{out}}}{E_N^{\mathrm{in}}},
    \quad
    \widetilde R_N:=\frac{\widetilde E_N^{\mathrm{out}}}{\widetilde E_N^{\mathrm{in}}}.
\end{equation}
Because $D_1$ and $D_2$ are orthonormal, these are exactly the corresponding Euclidean errors in the coefficient vectors $p(t)$ and $q(t)$.

\begin{figure}[htbp]
  \centering
  \safeincludegraphics[width=0.94\linewidth]{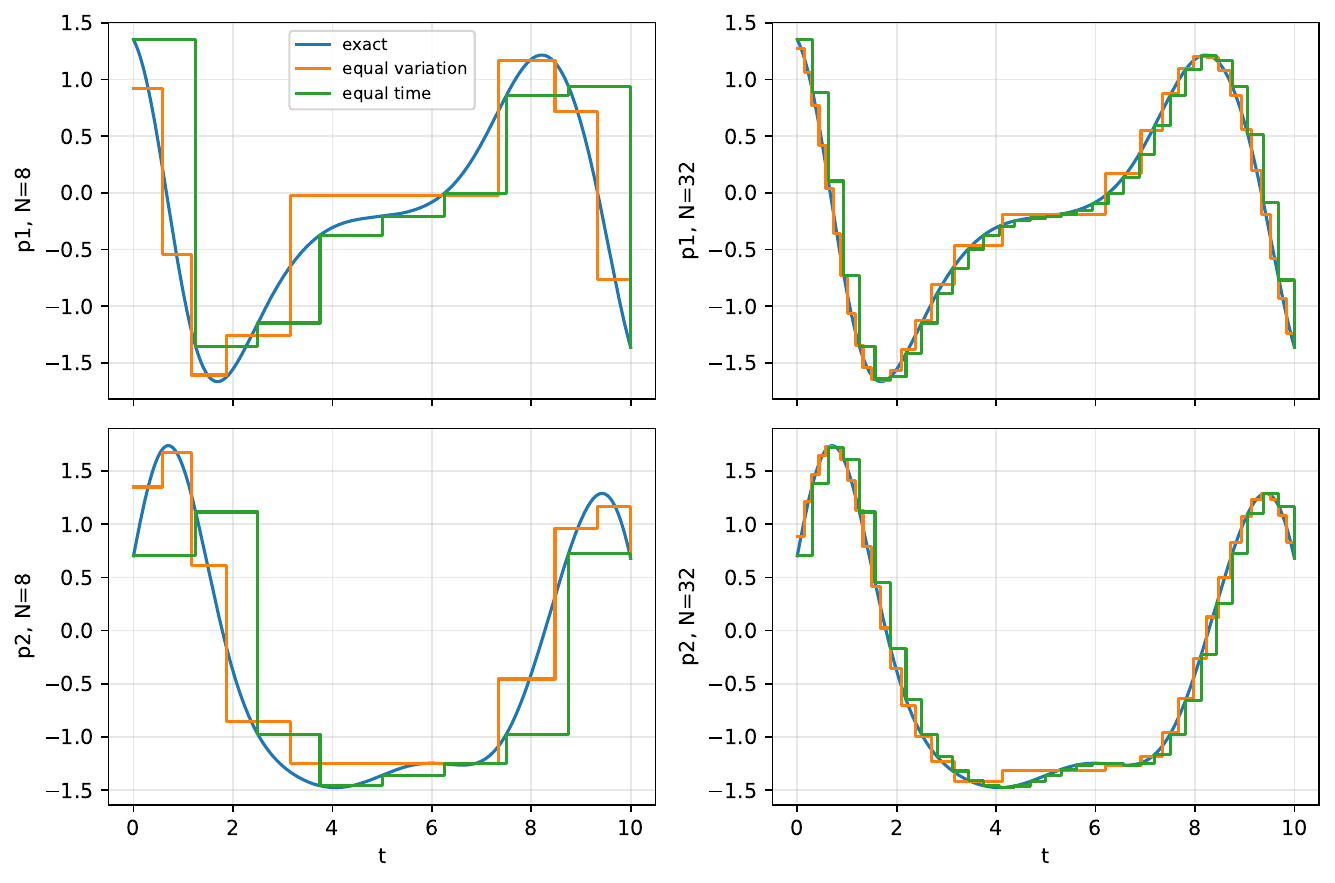}
  \caption{Input components $p_1(t)$ and $p_2(t)$ together with the equal-variation and equal-time input encoders for representative values of $N$.}
  \label{fig:nonprop-input-components}
\end{figure}

\begin{figure}[htbp]
  \centering
  \safeincludegraphics[width=0.86\linewidth]{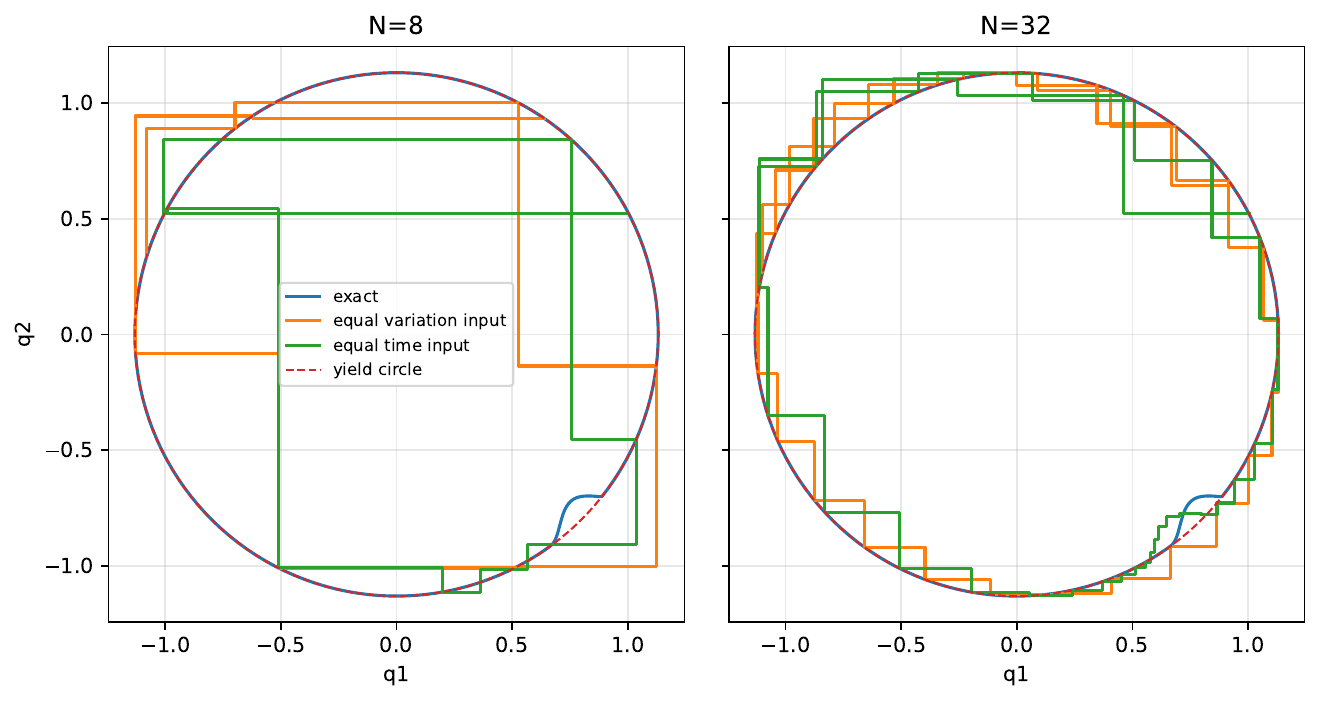}
  \caption{Input-first stress paths in the $(D_1,D_2)$-plane. The dashed curve is the yield circle $\partial B_{s_{\mathrm Y}}$.}
  \label{fig:nonprop-stress-paths}
\end{figure}

\begin{figure}[htbp]
  \centering
  \safeincludegraphics[width=0.94\linewidth]{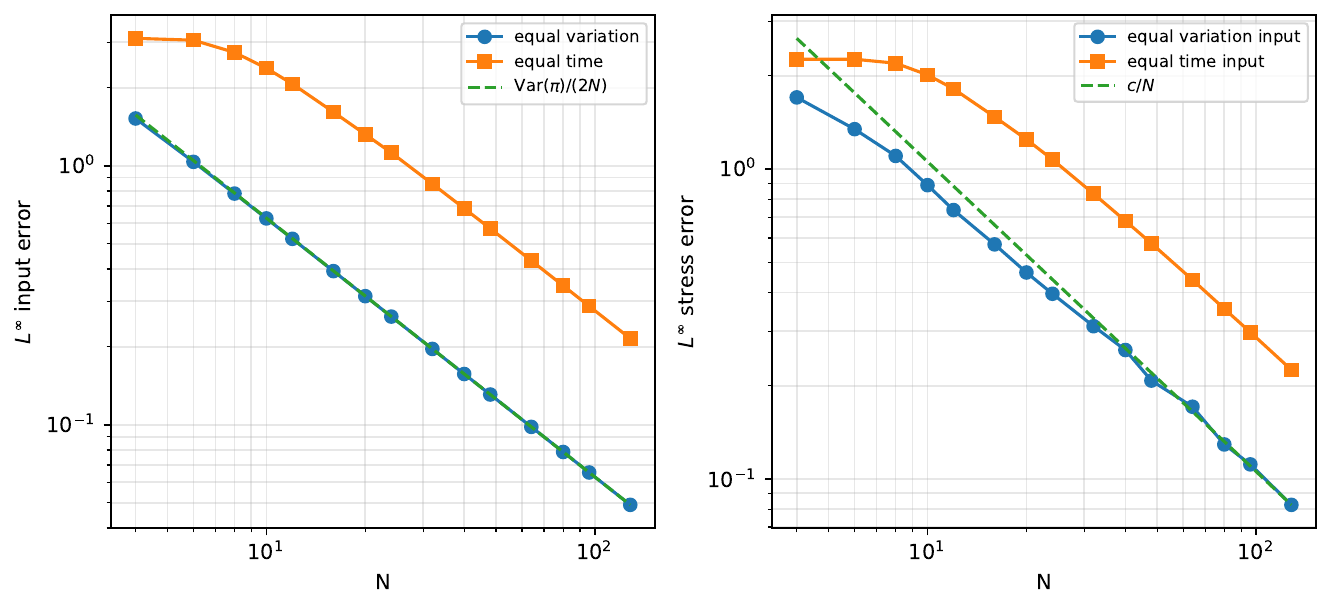}
  \caption{Input and input-first stress approximation errors versus $N$. The dashed line in the right panel is a $c/N$ reference slope.}
  \label{fig:nonprop-errors}
\end{figure}

\begin{figure}[htbp]
  \centering
  \safeincludegraphics[width=0.68\linewidth]{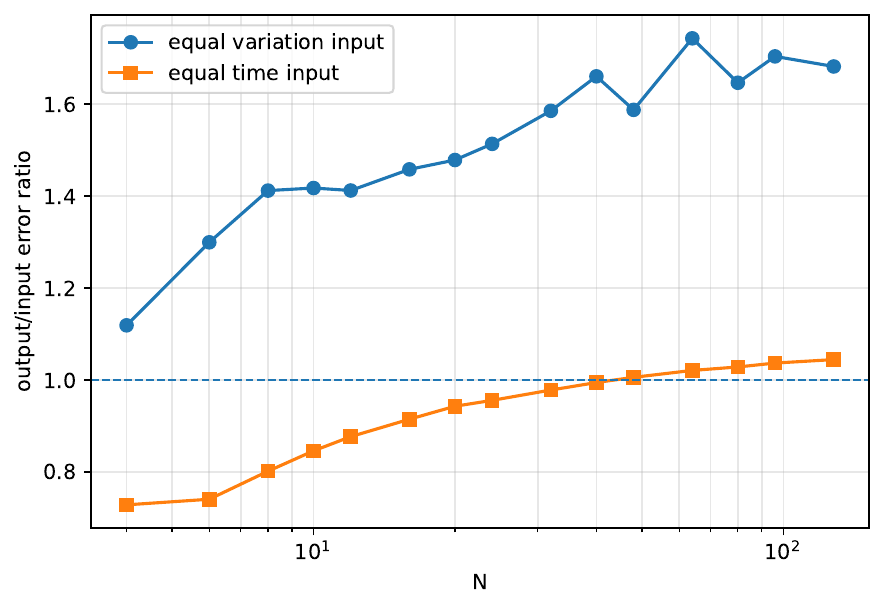}
  \caption{Ratio of input-first output error to input error. Values above $1$ indicate amplification of the $L^\infty$ input perturbation by the vector constitutive map.}
  \label{fig:nonprop-amplification}
\end{figure}

Figure~\ref{fig:nonprop-input-components} shows that the equal-variation input encoder concentrates its degrees of freedom near portions of the history where the combined radial and angular variation is largest. Figure~\ref{fig:nonprop-stress-paths} shows that the exact stress spends long intervals on the yield circle and moves tangentially along it as the loading direction rotates. This tangential sliding is absent in the proportional example and is a direct manifestation of nonproportionality. Figure~\ref{fig:nonprop-errors} shows that, for this example, the equal-variation input encoder continues to outperform the equal-time input encoder both at the input and at the input-first output level. The input-first vector effect is clearest in Figure~\ref{fig:nonprop-amplification}: the ratio $R_N$ can exceed $1$, so an $L^\infty$ input error cannot be pushed through the vector decoder by a non-expansive estimate.

We next repeat the same demonstration with the stress-aware code. Let $\sigma=\mathcal P(\pi)$ and set
\begin{equation} \label{eq:nonprop-stress-aware-code}
    h_N^\star(\pi):={S}_N^\star(\sigma),
    \qquad
    \sigma_N^{\mathrm{sa}}:=g_N^\star(h_N^\star(\pi))={S}_N^\star(\sigma).
\end{equation}
In coefficient form, this means that the vector stress path $q(t)$ is sampled by equal increments of its Euclidean total variation in the $(q_1,q_2)$-plane. Since every sampled value lies in the disk $B_{s_{\mathrm Y}}$, the decoder again acts as the identity on the code. The stress-aware error is
\begin{equation} \label{eq:nonprop-stress-aware-error}
    E_N^{\mathrm{sa}}
    :=\|\sigma-\sigma_N^{\mathrm{sa}}\|_{L^\infty(0,T;F)}
    \le
    \frac{1}{2N}\,\Var(\sigma;[0,T]).
\end{equation}

\begin{figure}[htbp]
  \centering
  \safeincludegraphics[width=0.86\linewidth]{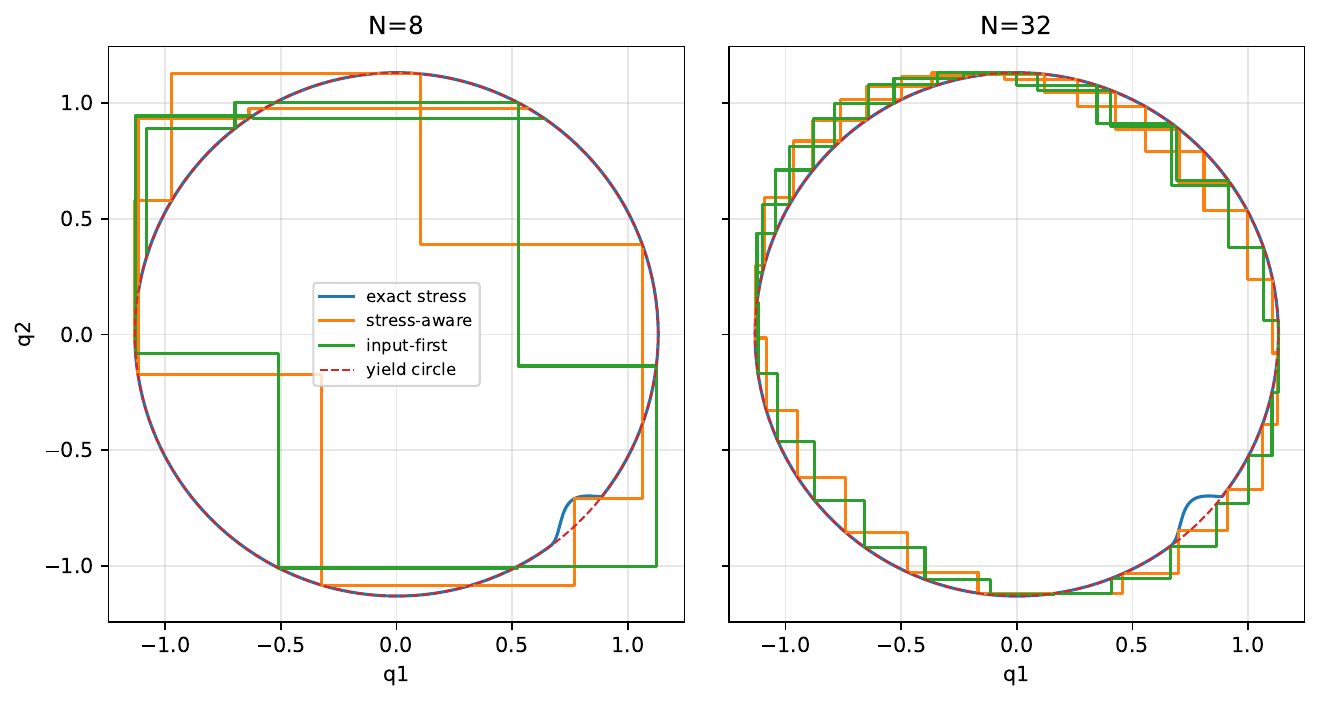}
  \caption{Nonproportional loading: exact stress path, input-first equal-variation decoded path, and stress-aware decoded path. The stress-aware samples track the variation of the stress trajectory on the yield disk rather than the variation of the input trajectory.}
  \label{fig:nonprop-stressaware-paths}
\end{figure}

\begin{figure}[htbp]
  \centering
  \safeincludegraphics[width=0.78\linewidth]{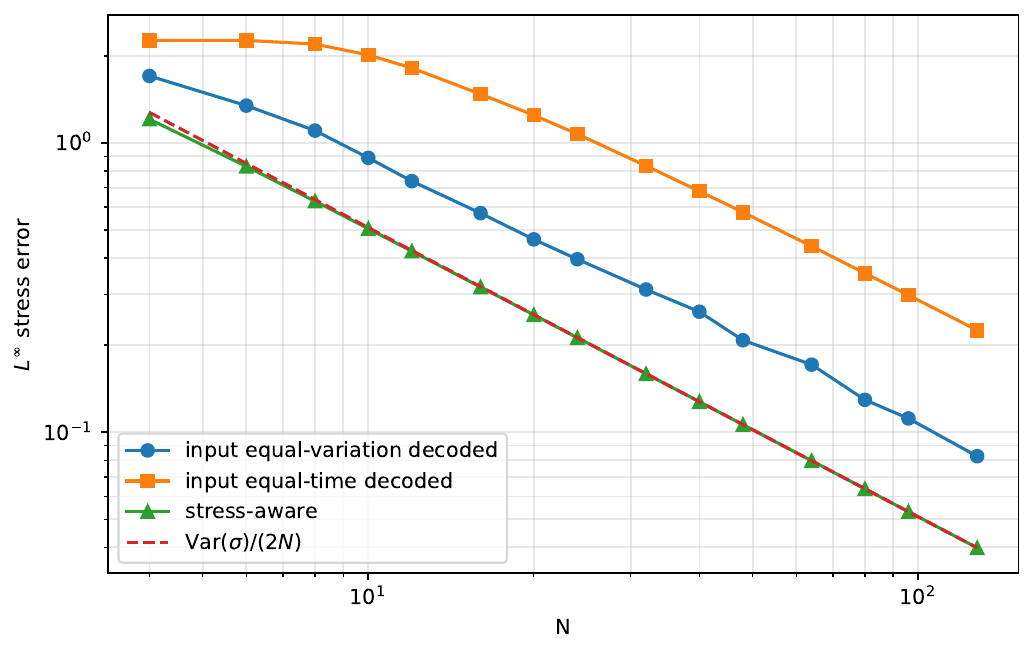}
  \caption{Nonproportional loading: comparison of input-first stress errors and stress-aware stress error. The dashed curve is the bound $\Var(\sigma;[0,T])/(2N)$.}
  \label{fig:nonprop-stressaware-errors}
\end{figure}

Figure~\ref{fig:nonprop-stressaware-paths} shows the geometric effect of stress-aware sampling in the genuinely vector setting. The input-first code follows the stress response generated by a compressed loading path; its sample locations are inherited from the variation of $p(t)$. The stress-aware code instead follows the computed stress path itself, including the portions of tangential motion on the yield circle. Figure~\ref{fig:nonprop-stressaware-errors} shows the corresponding stress errors. The stress-aware curve follows the stress-variation bound \eqref{eq:nonprop-stress-aware-error}, whereas the input-first curves are affected by the vector amplification already displayed in Figure~\ref{fig:nonprop-amplification}.

In summary, the nonproportional example demonstrates why the sharp vector constitutive theorem is formulated with stress-aware encoders. Equal-variation sampling remains the right principle, but the variable in which variation is measured matters. Sampling the input path gives a good reduced loading history; sampling the stress path gives the sharp constitutive code and avoids relying on an unavailable vector $L^\infty$ non-expansiveness property.

\clearpage
\section{Summary and concluding remarks}
\label{sec:conclusions}

This paper treats a local, material-point problem for rate-independent elastic-ideally-plastic response and its approximation by finite hereditary surrogates. The main conclusions are as follows:

\smallskip

\begin{itemize}
\item[(i)] Starting from the closest-point projection algorithm on step inputs, one obtains a hereditary constitutive operator on absolutely continuous driving histories. Equivalently, for $\pi\in W^{1,1}(0,T)$ the stress history remains in the elastic domain and satisfies a differential inclusion. In this $W^{1,1}$ setting, the constitutive map is causal, preserves the elastic domain, contracts total variation, and satisfies a $BV$-to-$L^\infty$ stability estimate.

\item[(ii)] For approximation of the \emph{input} history by right-continuous step surrogates with at most $N$ constant pieces, the minimax problem on absolutely continuous histories admits a sharp solution. The optimal input encoder is the sampled equal-variation encoder: one partitions the history by equal increments of cumulative variation and samples at the corresponding variation midpoints. The normalized worst-case input error is exactly $(2N)^{-1}$.

\item[(iii)] For stress-valued constitutive approximation, the sharp vector minimax theorem is obtained by using a constitutive, stress-aware encoder. The encoder first forms $\sigma=\mathcal P(\pi)$, then stores the equal-variation step approximation ${S}_N^\star(\sigma)$. Since step histories taking values in the elastic domain are fixed points of the discrete decoder, the decoded output is exactly this encoded stress history. Under $0\in\operatorname{int}C$, the normalized constitutive minimax error is again exactly $(2N)^{-1}$.

\item[(iv)] If one instead insists on an input-first strategy and estimates the output only through vector $BV$-to-$L^\infty$ stability, the resulting bounds are suboptimal. The sharp $L^\infty$ input estimate cannot be propagated through the vector decoder because the vector play operator is not $L^\infty$-nonexpansive. This explains why the $BV$-stability reduction is useful as an a posteriori or input-first estimate, but obsolete as a way of proving the sharp constitutive minimax bound.

\item[(v)] The scalar case remains special in a different sense. For one-dimensional loading, the complementary variable is $L^\infty$-nonexpansive with respect to the driving history, and the complementary-variable minimax problem is sharp with the purely input-based equal-variation encoder. Thus the scalar result complements, rather than replaces, the stress-aware vector theorem.
\end{itemize}

\smallskip

The principal structural message is therefore clear: for hereditary rate-independent response, the natural variable for compression is cumulative variation. For the input problem, cumulative variation is computed from the driving history. For the sharp stress-valued constitutive problem, it is computed from the stress history itself. The encoder/decoder viewpoint accommodates both choices: an input encoder compresses prescribed data before constitutive evaluation, whereas a stress-aware constitutive encoder compresses the material response and is decoded by the same hereditary law.

This observation has direct implications for Data-Driven inelasticity and related computational architectures, where the central challenge is to populate and organize history-dependent material data efficiently \cite{KirchdoerferOrtiz2016, StainierLeygueOrtizEtAl2019, EggersmannStainierOrtizReese2021}. The present analysis suggests two complementary sampling principles. Input histories should be sampled approximately uniformly in accumulated input variation when the goal is to compress prescribed strain histories. Constitutive databases, by contrast, should organize admissible stress/strain histories by accumulated response variation when the goal is to approximate the material response itself. This viewpoint is consistent with efficient phase-space sampling and search strategies in Data-Driven computational mechanics \cite{EggersmannStainierOrtizReese2021}, and it parallels analogous compression principles for linear hereditary response \cite{SalahshoorOrtiz2022}.

The hereditary formulation also suggests a natural connection with global solution strategies based on equivalent, or fictitious, body forces \cite{Mendelson1968, Mura1987}, also known as the method of initial stress or initial stiffness \cite{ZienkiewiczValliappanKing1969, Thomas1984, SloanShengAbbo2000}. Once the local constitutive update is written as a history operator, the inelastic contribution may be viewed as a history-dependent source term within an elastic Green-operator or equivalent-body-force iteration. Equal-variation compression may then be applied either to the imposed local driving histories or, after local updates, to the resulting source/stress histories. The asynchronous local stepping induced by equal-variation sampling may in turn be handled using ideas related to asynchronous variational integrators in dynamics \cite{LewMarsdenOrtizWest2003}. A rigorous realization of this program for initial-boundary-value problems lies beyond the present local theory, but the compatibility is conceptually clear.

Several directions for further work follow naturally. One is the extension from ideal plasticity to hardening, softening, and more general rate-dependent models such as elasto-viscoplasticity. For linear rate dependency, the return mapping $P_C$ is frustrated by viscosity and falls short of $C$ exponentially in time \cite{OrtizPinskyTaylor1983}; rate dependency of the power-law type \cite{OrtizStainier1999} suggests recourse to $L^p$ spaces in place of the present $BV$ framework. A second direction is the approximation of genuinely nonlocal history operators, where the encoded variable is no longer a single local path but a field of interacting histories. A third is the development of Data-Driven solvers that use equal-variation sampling as a principled rule for offline history-database construction. A fourth is a rigorous extension of the present local theory to spatially distributed initial-boundary-value problems. The sharp input theorem, the sharp stress-aware vector constitutive theorem, and the scalar complementary-variable theorem obtained here provide precise benchmarks against which such developments may be measured.

\section*{Acknowledgements}

The financial support of the \emph{Centre Internacional de M\`etodes Num\`erics a l'Enginyeria} (CIMNE) of the \emph{Universitat Polit\`ecnica de Catalunya} (UPC), Spain, through the \emph{UNESCO Chair in Numerical Methods in Engineering} is gratefully acknowledged. The work of the second author is supported by grants PID2023-151823NB-I00, and  SBPLY/23/180225/000023.

\end{document}